\documentclass{amsart}\def\distribute{}
\ifx\distribute\undefined
\date{}
\else
\date{July 09, 2015}
\fi
\usepackage{graphicx} 
\usepackage{amssymb}
\usepackage{verbatim,enumerate}
\usepackage[dvipsnames]{xcolor}
\title[
    An index formula for a bundle homomorphism 
]{%
    An index formula for a bundle homomorphism of 
    the tangent bundle into a vector bundle of the same rank,\\
    and its applications
}
\author[K. Saji]{Kentaro Saji}
\address[Saji]{%
  Department of Mathematics,
  Faculty of Science,
  Kobe University,
  Rokko, Kobe 657-8501}
\email{saji@math.kobe-u.ac.jp}

\author[M. Umehara]{Masaaki Umehara}
\address[Umehara]{%
  Department of Mathematical and Computing Science,
  Tokyo Institute of Technology,
  O-okayama Meguro-ku,
  Tokyo 152-8552, Japan
}
\email{umehara@is.titech.ac.jp}

\author[K. Yamada]{Kotaro Yamada}
\address[Yamada]{%
  Department of Mathematics,
  Tokyo Institute of Technology,
  O-okayama, Meguro-ku, Tokyo 152-8551, Japan%
}
\email{kotaro@math.titech.ac.jp}
\ifx\distribute\undefined
\relax
\else
\thanks{%
  The first author was
  partly supported by Grant-in-Aid for Scientific Research
  (C) No.\ 26400087
  from the Japan Society for the Promotion of Science,
  the second author by (A) No.\ 26247005 and
  the third author by (C) No.\ 26400006 from the Japan
  Society for the Promotion of Science.
}
\fi
\usepackage{amsthm}
\theoremstyle{plain}
 \newtheorem{theorem}{Theorem}[section]
 \newtheorem{proposition}[theorem]{Proposition}
 \newtheorem{prop}[theorem]{Proposition}
 
 \newtheorem{lemma}[theorem]{Lemma}
 \newtheorem{corollary}[theorem]{Corollary}
 
\theoremstyle{definition}
 \newtheorem{definition}[theorem]{Definition}
\theoremstyle{remark}
 \newtheorem{remark}[theorem]{Remark}
 \newtheorem*{remark*}{Remark}
 \newtheorem{example}[theorem]{Example}
 \newtheorem*{acknowledgements}{Acknowledgements}
\numberwithin{equation}{section}
\renewcommand{\theenumi}{{\rm(\arabic{enumi})}}
\renewcommand{\labelenumi}{\theenumi}

\newcommand{\vect}[1]{\boldsymbol{#1}}
\newcommand{\R}{\boldsymbol{R}}

\newcommand{\Z}{\boldsymbol{Z}}
\renewcommand{\phi}{\varphi}
\newcommand{\sgn}{\operatorname{sgn}}
\newcommand{\ind}{\operatorname{ind}}
\newcommand{\grad}{\operatorname{grad}}
\newcommand{\inner}[2]{\left\langle{#1},{#2}\right\rangle}
\newcommand{\Ker}{\operatorname{Ker}}

\newcommand{\E}{\mathcal{E}}
\newcommand{\A}{A}

\newcommand{\n}{\mathcal{N}}
\newcommand{\id}{\operatorname{id}}

\newcommand{\mb}[1]{{\mathbf #1}}
\newcommand{\norm}{\vect{n}}
\newcommand{\pmt}[1]{{\begin{pmatrix} #1  \end{pmatrix}}}
\renewcommand{\AA}{\mathfrak{A}}
\newcommand{\X}{\mathfrak{X}}

\newcommand{\J}{\mathcal{J}}
\newcommand{\N}{\mathcal{N}}
\newcommand{\T}{\mathcal{T}}
\begin{document}
\begin{abstract}
 In a previous work, the authors introduced
 the notion of \lq coherent tangent bundle\rq, 
 which is useful for giving a treatment of singularities of
 smooth maps without ambient spaces.
 Two different types of Gauss-Bonnet formulas on coherent 
 tangent bundles on $2$-dimensional manifolds were proven, 
 and several applications to surface theory were given. 

 Let $M^n$ ($n\ge 2$) be an oriented compact $n$-manifold 
 without boundary and $TM^n$ its tangent bundle.
 Let $\E$ be a vector bundle of rank $n$ over $M^n$, and 
 $\phi:TM^n\to \E$ an oriented  vector bundle homomorphism. 
 In this paper, we show that one of these two Gauss-Bonnet 
 formulas can be generalized to an index formula for the 
 bundle homomorphism $\phi$ under the assumption 
 that $\phi$ admits only certain kinds of generic 
 singularities.

 We shall give several applications to hypersurface theory.
 Moreover, as an application for intrinsic geometry,
 we also give a characterization of the
 class of positive semi-definite metrics
 (called Kossowski metrics)
 which can be realized as the induced metrics of
 the coherent tangent bundles.
\end{abstract}
\maketitle
\ifx\distribute\undefined
\footnotetext{%
  The first author was
  partly supported by Grant-in-Aid for Scientific Research
  (C) No.\ 26400087
  from the Japan Society for the Promotion of Science,
  the second author by (A) No.\ 26247005 and
  the third author by (C) No.\ 26400006 from the Japan
  Society for the Promotion of Science.
}%
\fi
\section{Introduction}
Let $M^{n}$ be an oriented closed $n$-manifold
and $(\E,\inner{~}{~},D)$ an oriented vector bundle 
of rank $n$ having inner product $\inner{~}{~}$ and 
a metric connection $D$, that is 
\[
  X\inner{\xi_1}{\xi_2}
        =\inner{D_X\xi_1}{\xi_2}+\inner{\xi_1}{D_X\xi_2}
\]
holds, where $\xi_{i}$ ($i=1,2$) are sections of $M^n$ into 
$\E$ and $X$ is a vector field of $M^n$.
A bundle homomorphism
\[
  \phi:TM^n\to (\E,\inner{~}{~},D)
\]
is called a {\it coherent tangent bundle\/} if it satisfies
\begin{equation}\label{eq:torsion}
   D_X\phi(Y)-D_Y\phi(X)=\phi([X,Y])
\end{equation}
for any two vector fields $X,Y$ on $M^n$.
When $n=2$, the authors proved in \cite{SUY0} and \cite{SUY1} 
that the two different types of Gauss-Bonnet formulas
($\chi_{\E}^{}$ is the Euler characteristic of the
 oriented vector bundle $\E$)
\begin{align}
  (\chi_{\E}^{}=)&\frac1{2\pi}\int_{M^2}K\, d\hat A_\phi=
        \chi(M^2_+)-\chi(M^2_-)
              +S^+_\phi-S^-_\phi, 
  \label{eq:B}\\
  2\pi\chi(M^2)&=
  \int_{M^2}K\, dA_\phi+2\int_{\Sigma_\phi} \kappa_\phi\, d\tau_\phi, 
  \label{eq:A}
\end{align}%
under the assumption that the singular set of $\phi$ 
consists of $\A_2$-points  and  $\A_3$-points, where
$K$ is the Gaussian curvature of the induced metric 
$ds^2_{\varphi}=\phi^*\inner{~}{~}$,
the two subsets $M^2_\pm$ are defined in \eqref{eq:pm}, 
$d\tau_\phi$ is the length element on the  $\phi$-singular 
set with respect to $ds^2_\phi$,
and $S^\pm_\phi$ are the numbers of positive and negative 
$\A_3$-points of $\phi$, respectively.  
If $f:M^2\to \R^3$ 
is a wave front, and $(\phi:=)df:TM^2\to \E$
is the bundle homomorphism induced by $f$,
then $\A_2$-points (resp.\ $\A_3$-points)
correspond to cuspidal edges (resp.\ swallowtails).
The precise definition of $\A_2$ or $\A_3$-points are
given in Section \ref{sec:prelim}.
The authors gave several applications of this formula
in \cite{SUY5} and \cite{SUY6} for surfaces in $\R^3$.

We remark that the second formula \eqref{eq:A} depends 
on the metric connection $D$,
but the first formula \eqref{eq:B} does not need 
information about the inner product.
So it is natural to expect that one can
extend the formula \eqref{eq:B}
to higher dimensional cases. The purpose of this paper is
to accomplish this for even dimensional manifolds without assuming 
condition \eqref{eq:torsion} as follows:
let $\phi:TM^n\to \E$ be a homomorphism between the tangent bundle 
$TM^n$ and an oriented vector bundle $\E$ of rank $n$ on $M^n$. 
Suppose that $\phi$ admits only $\A_k$-singular points 
(the definition of $A_k$-singular points ($k=2,\dots,n$) 
is given in Section~\ref{sec:prelim}).
We denote by $\AA_k$ ($k=2,\dots,n$)
the set of $A_k$-singular points.
When $k$ is odd, we can define the positivity and
negativity of $A_k$-points (see Section \ref{sec:char}). 
We denote by $\AA_k^+$ (resp.\ by $\AA_k^-$) the set of
positive (resp.\ negative) $A_k$-singular points.
When $n=2m$ is an even number, 
the Euler characteristic $\chi^{}_\E$ of the vector
bundle $\E$ satisfies the following formula
\begin{equation}\label{eq:main}
   \chi^{}_\E =\chi(M^n_+)-\chi(M^n_-)+
   \sum_{j=1}^m \biggl( \chi(\AA_{2j+1}^+)
                       -\chi(\AA_{2j+1}^-)\biggr),
\end{equation}
where $\chi(M^n_+)$ (resp.\ $\chi(M^n_-)$) is the 
Euler characteristic of the subset $M^n_+$ (resp.\ $M^n_-$) of 
$M^n$ at which the co-orientation induced by $\phi$ is 
(resp.\ is not) compatible with the orientation of $TM^n$ 
(cf.\ \eqref{eq:pm}),
the number $\chi(\AA^+_{2j+1})$ (resp.\ $\chi(\AA^-_{2j+1})$)
is the Euler characteristic of $\AA^+_{2j+1}$
(resp.\ $\AA^-_{2j+1}$).
In particular, $\chi(\AA^+_{2m+1})$ (resp.\ $\chi(\AA^-_{2m+1})$)
is equal to the number $\#\AA^+_{2m+1}$ (resp.\ $\#\AA^-_{2m+1}$)
of positive (resp.\ negative)
$A_{2m+1}$-points (cf.\ Definition~\ref{fact:ojm}).
For example, the formulas for $n=2$, $4$ are given by
\begin{align}
\label{eq:i2}
     \chi^{}_\E&=\chi(M^{2}_+)-\chi(M^{2}_-)+\#\AA^+_3-\#\AA^-_3, 
 \\
\label{eq:i4}     
     \chi^{}_\E&=\chi(M^{4}_+)-\chi(M^{4}_-)+
                 \chi(\AA^+_3)-\chi(\AA^-_3)+\#\AA^+_5-\#\AA^-_5.
\end{align}
Formula \eqref{eq:i2} is a generalization of \eqref{eq:A}.
As pointed out by Saeki and Sakuma in \cite{SS},
any closed orientable $4$-manifold with
vanishing signature admits $C^\infty$-maps
into $\R^4$ having only fold or cusp singularities.
The $\Z_2$-version of our formula \eqref{eq:main} 
was given by Levine \cite{L2} (see \cite[Remark 3.12]{Sa}).
If we set $\phi$ to be the derivative of a Morin map 
$f:M^n\to N^n$, then we get \eqref{eq:Quine},
which is proved by Nakai \cite{N} and 
Dutertre-Fukui \cite{DF}.
Index formulas in $\Z_2$-coefficients 
for globally defined Morin maps  $f:M^n\to N^p$ ($n\ge p$)
are given by Fukuda \cite{Fu} and Saeki \cite{Sa},
and formula \eqref{eq:main} is a generalization of them. 
Our proof is independent of those in \cite{N} and \cite{DF}. 
More precisely, we apply the Poincar\'e-Hopf index formula for 
sections of oriented vector bundles.
(In \cite{N} and \cite{DF}, Viro's integral calculus 
 \cite{V} is applied.)
Our index formula does not rely on ambient spaces,
and we can give applications even for a case without
ambient space (cf.\ Section \ref{sec:Kossowski}). 

In fact, one of the important applications of 
\eqref{eq:main} is for a certain class of
positive semi-definite metrics.
We define a class of positive semi-definite
metrics on manifolds called \lq Kossowski metrics\rq\, 
which was originally defined by Kossowski \cite{K}.
The induced metrics of wave fronts in $\R^{n+1}$
which admit at most $A_{k+1}$-singularities
($k=1,\dots,n$) are all Kossowski metrics.
Conversely, Kossowski \cite{K} showed that germs of
real analytic generic Kossowski metrics on
$2$-manifolds can be realized as the first fundamental forms
of wave fronts in $\R^3$. 

Let $(\E,\phi,\inner{~}{~},D)$ be a coherent tangent bundle 
over an $n$-manifold $M^n$ then
the pull back of $\inner{~}{~}$ by $\phi$ gives a Kossowski 
metric on $M^n$ whenever $\phi$ admits at most 
non-degenerate singular points 
(Proposition~\ref{prop:frontal22}).
The converse assertion for $n=2$
was proved in  \cite{HHNSUY}.
In this paper, we generalize this for $n\ge 3$,
namely, we show that each Kossowski metric 
$ds^2$ induces a coherent
tangent bundle $(\E,\phi,\inner{~}{~},D)$
such that $ds^2=\phi^*\inner{~}{~}$
and the pull-back of the connection $D$
by $\phi$ coincides with the Levi-Civita connection
on the regular set of $\phi$ (cf.\ Theorem~\ref{thm:kossowski}).
We then get  
an index formula (cf.\ Corollary~\ref{cor:morin-metric})
for Kossowski metrics
on compact manifolds
admitting at most $A_{k+1}$ singularities
($k=1,\dots,n$).

To give other applications of 
formula \eqref{eq:main}, the case of 
$\E=TM^n$ is important.
An arbitrarily given bundle automorphism
$\phi:TM^n\to TM^n$ can be identified with
the set of $(1,1)$-tensors on $M^n$ ($n=2m$),
and \eqref{eq:main} reduces to the
following identity:  
\begin{equation}\label{eq:BW}
   2\chi(M^n_-)=
    \sum_{j=1}^m 
    \biggl(
      \chi(\AA_{2j+1}^+)-\chi(\AA_{2j+1}^-)
    \biggr).
\end{equation}
In Section~\ref{sec:appl}, we give several applications of
\eqref{eq:main} and \eqref{eq:BW} for 
geometry of hypersurfaces.

The paper is organized as follows:
in Section~\ref{sec:prelim}, 
we give a precise definition of
$\A_k$-singularities.
In Section~\ref{sec:char}, the well-definedness
of the positivity and negativity of odd order
$\A_{2k+1}$-singular points is shown. 
Moreover, we define 
characteristic vector fields with
respect  to the homomorphism
$\phi:TM^n\to \E$ and show the existence of 
such a vector field $X$ defined on $M^n$.
It is well-known that the sum of all 
indices of zeros of a generic 
section $Y$ of $\E$ is equal to
the Euler characteristic $\chi^{}_\E$ of the oriented vector bundle $\E$.
Since the section $Y:=\phi(X)$ of $\E$ 
has finitely many zeros,
it holds that
\begin{equation}\label{eq:Y}
 \chi^{}_\E
  =\sum_{p\in M^n\setminus \Sigma^{n-1}} \ind_p(Y)
  +\sum_{p\in \AA_2} \ind_p(Y)
  +\dots 
  +\sum_{p\in \AA_{n+1}} \ind_p(Y),
\end{equation}
where 
\[
   \Sigma^{n-1}:=\AA_2\cup \dots \cup  \AA_{n+1}
\]
is the singular set of $\phi$.
Using this, we prove \eqref{eq:main}
in Sections \ref{sec:even2} and \ref{sec:even4}.
In Section~\ref{sec:appl}, we prove Theorems \ref{thm:A}
and \ref{thm:B}.
Several other applications are 
given in Section~\ref{sec:appl} and Section \ref{sec:Kossowski}.

\section{Preliminaries}\label{sec:prelim}
Let $M^n$ be an oriented $n$-manifold
and $\phi:TM^n\to \E$ a bundle homomorphism
between the tangent bundle $TM^{n}$ and 
a vector bundle $\E$ of rank $n$.
Then a point $p\in M^n$ is called 
a {\it singular point\/}
if the linear map $\phi_p:T_pM^n\to \E_p$ 
has a non-trivial kernel, 
where $\E_p$ is the fiber of $\E$ at $p$.
Since $M^n$ is oriented, we can take
a non-vanishing $n$-form $\Omega$ defined on $M^n$
which is compatible with the orientation of $M^n$.
We call $\Omega$ an {\it orientation\/} of $M^n$.

On the other hand, 
$\E$ is locally oriented, that is,
there is a non-vanishing section $\mu$
of the determinant line bundle of the dual 
bundle $\E^*$ of $\E$
defined on a neighborhood $U(\subset M^n)$ of 
a given point $p\in M^n$.
We call $\mu$ a {\it local orientation} of $\E$.

Then there is a (unique) $C^\infty$-function 
$\lambda:U\to \R$ such that
\begin{equation}\label{eq:lambda0}
    \phi^*\mu=\lambda\, \Omega,
\end{equation}
on $U$,
where $\phi^*\mu$ is the pull-back of $\mu$ by $\phi$.
A point $q\in U$ is a singular point if and only if 
$\lambda(q)=0$. A singular point $q\in M^n$ 
is called {\it non-degenerate\/}
if the exterior derivative $d\lambda$ does not
vanish at $q$.
The bundle homomorphism 
$\phi$ is called {\it non-degenerate\/}
if all the singular points are non-degenerate.
If $\varphi$ is non-degenerate, 
the singular set
\[
     \Sigma^{n-1}:=\{q\in M^n\,;\, \Ker(\phi_q)\ne \{0\}\}
\]
is an embedded hypersurface of $M^n$, where
$\Ker(\phi_q)$ is the kernel of the
linear mapping $\phi_q:T_qM^{n}\to \E_{\phi(q)}$. 

\begin{definition}\label{def:phi-function}
 Let $U$ be an open subset of $M^n$.
 A function $h:U\to \R$ is called a
 {\it $\phi$-function\/} if 
 there exists a  $C^\infty$-function 
 $\sigma:U\to \R\setminus\{0\}$
 such that 
 \begin{equation}\label{eq:phi-function}
   h=\sigma \lambda
 \end{equation}
 on $U$, where $\lambda$ is the 
 function as in \eqref{eq:lambda0}.
\end{definition}

Of course, $\lambda$ itself is a $\phi$-function.
However, $\lambda$ depends on the choice of 
$\Omega$ and $\mu$, and this ambiguity is just 
corresponding to the choice of $\phi$-functions.
In the following discussion, we may replace $\lambda$
by an arbitrarily fixed $\phi$-function.

Suppose that $\phi$ is non-degenerate.
Then the kernel of $\phi$ 
at each singular point $p\in \Sigma^{n-1}$ is of dimension $1$.
In particular, 
there exists a smooth vector field
$\tilde \eta$ defined on a sufficiently small 
neighborhood $U(\subset M^n)$ of $p$
such that the restriction 
\[
    \eta:=\tilde \eta|_{U\cap \Sigma^{n-1}}
\]
has the property that $\eta_q$ is
the generator of the kernel of $\phi_q$ 
for each $q\in U\cap \Sigma^{n-1}$.
We call $\eta$ a {\it null vector field\/} and
$\tilde \eta$ an {\it extended null vector field\/}
(cf.\ \cite[p.\ 733]{SUY3}).
For a given
extended null vector field $\tilde \eta$,
we often denote by $\eta$
the restriction of $\tilde \eta$
to $\Sigma^{n-1}$.
We set 
\begin{equation}\label{eq:lambda-der-1}
   \tilde\eta\lambda :=
   d\lambda(\tilde\eta),\qquad
   \tilde\eta\tilde\eta\lambda (= \tilde\eta^{2}\lambda):=
   d\bigl(\tilde\eta\lambda\bigr)(\tilde\eta),
\end{equation}
and 
\begin{equation}\label{eq:lambda-der}
    \tilde\eta^{k+1}\lambda:=d\bigl(\tilde\eta^{k}
\lambda\bigr)(\tilde \eta)
    \qquad (k=0,1,2,\dots),
\end{equation}
inductively.
As a convention, we set $\tilde\eta^{0}\lambda:=\lambda$.

\begin{definition}\label{fact:ojm}
 Let $\varphi\colon{}TM^n\to\E$ be a non-degenerate 
 bundle homomorphism
 and $\Sigma^{n-1}$ its singular set.
 A point $p\in \Sigma^{n-1}$ is an $\A_{k+1}$-point
 $(1\leq k\leq n)$ if 
 \begin{enumerate}
  \item\label{item:ojm:1}
       $\lambda(p)=\tilde\eta\lambda(p)=\dots =\tilde\eta^{k-1}\lambda(p)=0$,
       $\tilde\eta^{k}\lambda(p)\ne 0$,
  \item\label{item:ojm:2} 
       and the Jacobi matrix of the $\R^k$-valued
       $C^\infty$-function 
       \[
	\Lambda:=(\lambda,\tilde\eta\lambda,\dots,\tilde\eta^{k-1}\lambda)
       \]
       is of rank $k$ at $p$.
 \end{enumerate}
We denote by $\AA_{k+1}$ the set 
of $A_{k+1}$-points on $M^n$.
\end{definition}

Suppose that
$\varphi\colon{}TM^n\to\E$ is a non-degenerate 
bundle homomorphism.
If
$k=1$ (namely, for $\A_2$-points), then
$d\Lambda=d\lambda$ and 
the condition \ref{item:ojm:2} of Definition~\ref{fact:ojm} is
automatically satisfied.
Moreover, if $k=2$, the condition \ref{item:ojm:2} also
follows from \ref{item:ojm:1}. In fact,
the two differential forms
$d\lambda$ and $d(\tilde \eta\lambda)$ 
are linearly independent
at $p$,  since $d\lambda(p)\ne 0$,
$\tilde \eta \lambda(p)=0$ 
and $\tilde \eta^2\lambda(p)\ne 0$. 
In other words, the second condition of 
Definition~\ref{fact:ojm}
comes into effect only for $k\ge 3$ 
if $\phi$ is non-degenerate.

\medskip
Let $\phi:TM^n\to \E$ be a bundle homomorphism.
Suppose that $\phi:TM^n\to \E$ is non-degenerate
and the singular set $\Sigma^{n-1}$ is non-empty. 
Then the map
\begin{equation}\label{eq:quotient}
\hat \phi:T\Sigma^{n-1}\to \hat \E:=
\phi(TM^n|_{\Sigma^{n-1}})
\end{equation}
is induced.

\begin{proposition}
\label{prop:def}
 In this situation, $\hat \E$ is a vector bundle
 of rank $n-1$ on $\Sigma^{n-1}$,
 and $\hat \phi:T\Sigma^{n-1}\to \hat \E$
 is a bundle homomorphism.
\end{proposition}

We call $\hat \phi$ the {\it reduction\/} of $\phi$. 
By Proposition \ref{prop:def},
$\hat \E$ is a subbundle of codimension one
of $\E$.

\begin{proof}
 We fix a point $p\in \Sigma^{n-1}$ arbitrarily.
 It is sufficient to show the existence
 of linearly independent
 local sections $s_1,\dots,s_{n-1}$
 of $\hat \E$
 defined on a neighborhood of $p$
 in $\Sigma^{n-1}$.
 Since $\phi$ is non-degenerate,
 there exists an extended null
 vector field $\tilde \eta$ defined
 on a local coordinate neighborhood
 $(U;x_1,\dots,x_n)$ centered at $p$.
 Without loss of generality,
 we may assume that
 $\eta={\partial}/{\partial x_n}$ 
 holds at $p$.
 Since the kernel of $\phi$ is one dimensional,
 $s_j:=\phi\left({\partial}/{\partial x_j}\right)$
 ($j=1,\dots,n-1$)
 has the desired property.
\end{proof}

The following two assertions
(cf.\ Theorems \ref{thm:induction1}
and \ref{thm:induction2})
 gives
fundamental properties of the reduction homomorphism.

\begin{theorem}\label{thm:induction1}
 Let $\Sigma^{n-1}$
 be the singular set of a
 non-degenerate bundle homomorphism
 $\phi:TM^n\to \E$.
 Let $\lambda:U\to \R$ 
 and
 $\tilde \eta$
 be a $\phi$-function and
 an extended null vector field
 defined on an open subset
 $U(\subset M^n)$, respectively.
 Then the following assertions hold{\rm:}
 \begin{enumerate}
  \item\label{item:ind1} 
        The singular set $\Sigma^{n-2}\cap U$ 
        of $\hat \phi|_U$
	satisfies
	\[
	   \Sigma^{n-2}\cap U=\{q\in \Sigma^{n-1}
                  \cap U\,;\,
	   \tilde \eta_q\in T_q \Sigma^{n-1}\}
	  =\{q\in \Sigma^{n-1}\cap U\,;\, \lambda
	       =\tilde\eta\lambda=0\}.
	\]
  \item\label{item:ind2}
       $\tilde\eta\lambda$ 
       is a $\hat \phi$-function 
       defined on  $\Sigma^{n-1}\cap U$. 
 \end{enumerate}
\end{theorem}

\begin{proof}
 It can be easily checked that
 $\tilde \eta_q\in T_q\Sigma^{n-1}$ if and only if 
 $q\in \Sigma^{n-2}$ for each 
 $q\in \Sigma^{n-1}\cap U$.
 Thus, we get the equality
 \[
    \Sigma^{n-2}\cap U
     =\{q\in \Sigma^{n-1}\cap U
          \,;\,\tilde \eta_q\in T_q \Sigma^{n-1}\}
     =
      \{q\in \Sigma^{n-1}\cap U
            \,;\,\tilde\eta\lambda(q)=0\},
 \]
 proving the assertion \ref{item:ind1}.

 If $\hat \phi$ has no 
 singular points on $\Sigma^{n-1}\cap U$,
 then $\tilde \eta_q\not \in T_q\Sigma^{n-1}$
 for all $q\in \Sigma^{n-1}\cap U$, and thus
 $\tilde\eta \lambda$
 has no zeros, 
 so the assertion \ref{item:ind2}
 is trivially true.
 So we  may assume that 
 the singular set $\Sigma^{n-2}\cap U$ of
 $\hat \phi$ is not empty.

 We now fix a point $p\in \Sigma^{n-2}\cap U$,
 and take a local coordinate system
 $(V;y_1,\dots,y_n)$ centered at $p$
 such that $V\subset U$.
 Since $\phi$ is non-degenerate, 
 we may assume that
 $\partial\lambda/\partial y_1 \ne 0$ at $p$.
 By the implicit function theorem,
 there exists a function
 $y_1(y_2,\dots,y_n)$ such that
 $y_1(0,\dots,0)=0$ and
 \[
    \lambda\bigl(y_1(y_2,\dots,y_n),y_2,\dots,y_n\bigr)=0.
 \]
 If we set
\begin{equation}\label{eq:x-coord}
    x_1:=\lambda,\quad x_j:=y_j \qquad (j=2,\dots,n),
\end{equation}
then $(W;x_1,\dots,x_n)$ gives a new
 local coordinate system at $p$ if we choose $W(\subset V)$
 sufficiently small.
 We can write
\begin{equation}\label{eq:te}
\tilde\eta=b \partial_1+\sum_{j=2}^n c_j \partial_j
\end{equation}
on $\Sigma^{n-1}\cap W$, where we set 
 \[
   \partial_j:=\partial/\partial x_j\qquad  (j=1,2,\dots,n).
 \]
 Then we have that
 \begin{equation}\label{eq:lb}
  b(q)=\tilde\eta \lambda(q) \qquad (q\in \Sigma^{n-1}\cap W).
 \end{equation}
 Since $\tilde\eta\lambda(p)=0$ and $\tilde\eta(p)\ne 0$,
 we may assume that $c_2(p)\ne 0$ without loss of
 generality.
 If we set $\vect{e}_i:=\phi(\partial_i)$ ($i=1,2,\dots,n$),
 then 
 \begin{equation}
  \vect{e}_2
     =-\frac{\tilde\eta\lambda}{c_2}\vect{e}_1-\sum_{j=3}^{n}
       \frac{c_j}{c_2}
       \vect{e}_j
 \end{equation}
 holds on $\Sigma^{n-1}\cap W$
 for a sufficiently small $W$.
 We fix an inner product $\inner{~}{~}$
 on $\E$.  
 We can take a local unit section $\vect{u}$
 of $\E$ defined on 
 $\Sigma^{n-1}\cap W$ such that $\vect{u}$ is
 orthogonal to $\vect{e}_1,\dots,\vect{e}_n$.
 Then $\hat\E$ defined by \eqref{eq:quotient}
 is equal to
 the subbundle of $\E$ which is orthogonal
 to $\vect{u}$.
 Let $\mu$ be a local orientation of $\E$ on $W$.  
 It is obvious that
 $\vect{u}$, $\vect{e}_1$,
 $\vect{e}_3$,
 \dots,
 $\vect{e}_{n}$
 are linearly independent on $\Sigma^{n-1}\cap W$,
 and so we may assume that
 \[
    \delta:=\mu(\vect{u},\vect{e}_1,\vect{e}_3,
                    \dots,\vect{e}_{n})
 \]
 is a positive valued function on $\Sigma^{n-1}\cap W$. 
 Since $\hat\E$ is the
 subbundle of $\E$ which is orthogonal to $\vect{u}$, 
 \[
    \hat \mu(\vect{v}_1,\dots,\vect{v}_{n-1})
        :=\mu(\vect{u}_q,\vect{v}_1,\dots,\vect{v}_{n-1})
     \qquad (\vect{v}_1,\dots,\vect{v}_{n-1}\in \hat \E_q,\,\, q\in W)
 \]
 gives a local orientation of $\hat \E$, and
 a $\hat \phi$-function $\hat \lambda:W\to \R$ of
 $\hat\E$ is given by
 \begin{align*}
  \hat \lambda&:=
       \hat \mu(\vect{e}_2,\dots, \vect{e}_n)
           =\mu(\vect{u},\vect{e}_2,\dots, \vect{e}_n)\\
   &=-\frac{\tilde\eta\lambda}{c_2}
         \mu(\vect{u},\vect{e}_1,
                      \vect{e}_3,\dots,\vect{e}_{n})
         =-\frac{\delta}{c_2}\tilde\eta\lambda,
 \end{align*}
 which proves the assertion \ref{item:ind2}, since $p$ is an
 arbitrarily fixed point of $\Sigma^{n-1}\cap U$.
\end{proof}

Moreover, the following assertion holds.

\begin{theorem}\label{thm:induction2}
 Let $k$ be an integer satisfying
 $1\le k\le n$.
 Under the same assumptions as in
 Theorem \ref{thm:induction1},
 $p\in U$ is an $A_{k+1}$-point of $\phi$
 if and only if
 $p$ {\rm(}is a non-degenerate singular point
 of $\hat \phi$
 and{\rm)} 
 is an $A_k$-point of $\hat \phi$,
 where $A_1$-points mean regular points.
\end{theorem}

The restriction of
the null vector field $\tilde \eta$ to $\Sigma^{n-1}$
is not tangent to $\Sigma^{n-1}$ in general.
To prove Theorem \ref{thm:induction2},
we now construct an extended null vector
field $\tilde \zeta$ as a modification
of $\tilde \eta$ as follows:
as in the proof of Theorem \ref{thm:induction1},
we fix a point $p$.
Let $(W;x_1,\dots,x_n)$ be the local coordinate system
centered at $p$ given in the proof of 
Theorem \ref{thm:induction1}.
By \eqref{eq:te} and \eqref{eq:lb},
\begin{equation}\label{eq:zeta}
     \tilde\zeta=\tilde\eta-(\tilde\eta\lambda) \partial_1
\left(=\sum_{j=2}^n c_j \partial_j\right)
\end{equation}
gives an extended null vector field
of $\hat \phi$ on $\Sigma^{n-1}\cap W$.
Let $\mu_1$,\dots, $\mu_r$ be fixed smooth functions
on $W$. 
For two $C^\infty$-functions
$f$,  $g$ on $W$, we write 
\[
 f\equiv g \mod (\mu_1,\dots,\mu_r)
\]
if there exist $C^\infty$-functions $h_1$,\dots, $h_r$
defined on $W$ such that
\[
       f-g =h_1 \mu_1+\dots +h_r\mu_r.
\]
The following lemma is obvious:
\begin{lemma}\label{lem:mod}
 If $f\equiv g \mod (\mu_1,\dots,\mu_r)$,
 then it holds that
 \[
   \tilde \eta f\equiv \tilde \eta g \mod (\mu_1,\dots,\mu_r,
    \tilde \eta\mu_1,\dots,\tilde \eta\mu_r).
 \]
\end{lemma}

We prove the following assertion.
\begin{proposition}\label{prop:mod}
 The equalities
 \begin{equation}\label{eq:mod}
  \tilde\eta^{j+1}\lambda\equiv
   \tilde\zeta^{j}(\tilde\eta\lambda)\quad 
   \mod(\tilde\eta\lambda,\dots,\tilde\eta^{j}\lambda)
   \qquad (j=1,\dots,k-1)
 \end{equation}
 hold on $W$.
\end{proposition}

\begin{proof}
 We prove the assertion by induction on $j$.
 If $j=1$, we have that (cf.\ \eqref{eq:zeta})
 \[
    \tilde\zeta (\tilde\eta\lambda)
   =\bigl(\tilde\eta-(\tilde\eta\lambda) \partial_1\bigr)(\tilde \eta \lambda)\\
   =\tilde\eta(\tilde \eta \lambda)
       -\tilde \eta\lambda(\tilde \eta\lambda)_{x_1}\\
     \equiv
     \tilde\eta^{2}\lambda  \quad \mod (\tilde\eta\lambda).
 \]
 So we now assume that \eqref{eq:mod} 
 holds and consider the case of $j+1$.
 It holds that
 \[
   \tilde\zeta^{j+1}(\tilde \eta\lambda)
    =\tilde \zeta(\tilde \zeta^{j}(\tilde \eta \lambda))
    =\tilde \eta(\tilde \zeta^{j}(\tilde \eta \lambda))
    -\tilde \eta\lambda(\tilde \zeta^j(\tilde \eta\lambda))_{x_1}.
 \]
 In particular
 \begin{equation}\label{eq:induction1}
  \tilde \zeta^{j+1}(\tilde \eta\lambda)
   \equiv
   \tilde \eta(\tilde\zeta^{j}(\tilde \eta \lambda))
   \quad \mod (\tilde\eta \lambda).
 \end{equation}
 On the other hand,
 applying Lemma \ref{lem:mod} to \eqref{eq:mod}, we have
 \begin{equation}\label{eq:induction2}
  \tilde \eta(\tilde \zeta^j(\tilde \eta\lambda))
   \equiv \tilde \eta(\tilde \eta^{j+1}\lambda)
   \quad \mod(\tilde\eta\lambda,\dots,\tilde\eta^{j+1}\lambda).
 \end{equation}
 By \eqref{eq:induction1} and \eqref{eq:induction2},
 we get the assertion for $j+1$.
\end{proof}

\begin{proof}[Proof of Theorem \ref{thm:induction2}]
 By \ref{item:ind1} of Theorem \ref{thm:induction1},
 $p\in \Sigma^{n-1}$ is an $A_2$-point
 if and only if
 $\tilde \eta\lambda(p)\ne 0$.
 By \ref{item:ind2} of Theorem \ref{thm:induction1},
 $\tilde \eta\lambda$ is a $\hat \phi$-function,
 and thus $\tilde \eta\lambda(p)\ne 0$
 if and only if $p$ is a regular point of $\hat \phi$.
 This proves the assertion for $k=1$.
 So we now consider the case that $k\ge 2$.
 We set $\lambda_1:=\tilde \eta \lambda$.
 Since $k\ge 2$, we have
 \begin{equation}\label{eq:l12}
  \lambda(p)=\lambda_1(p)=0.
 \end{equation}
 Under this assumption \eqref{eq:l12}, 
 $p$ satisfies
 \ref{item:ojm:1} of Definition \ref{fact:ojm} if and only if
 \begin{equation}\label{eq:first0}
  \tilde \eta\lambda_1(p)=\cdots =
   \tilde \eta^{k-2} \lambda_1(p)=0,
   \quad \tilde \eta^{k-1} \lambda_1(p)\ne 0.
 \end{equation}
 By Proposition \ref{prop:mod},
 this is equivalent to the condition
 \begin{equation}\label{eq:first1}
  \tilde \zeta\lambda_1(p)=\cdots =
   \tilde \zeta^{k-2} \lambda_1(p)=0,
   \quad \tilde \zeta^{k-1} \lambda_1(p)\ne 0.
 \end{equation}
 On the other hand,
 we can take a local coordinate system
 $(x_1,\dots,x_n)$ centered at $p$
 such that (cf.\ \eqref{eq:x-coord})
 \begin{enumerate}
  \item
       $\lambda_{x_1}(p)\ne 0$ and
       $\lambda_{x_2}(p)=\cdots=\lambda_{x_n}(p)=0$,
  \item $(x_2,\dots,x_n)$ gives a local
	coordinate system of $\Sigma^{n-1}$ at $p$.
 \end{enumerate}
 The existence of this coordinate system
 yields that
 $p$ satisfies
 \ref{item:ojm:2} of Definition \ref{fact:ojm} if and only if
 the Jacobi matrix of the $\R^{k-1}$-valued
 $C^\infty$-function 
 \[
  \Lambda_1:=(\tilde\eta\lambda,\dots,\tilde\eta^{k-1}\lambda)
 =(\lambda_1,\tilde\eta \lambda_1,\dots,\tilde\eta^{k-2}\lambda_1)
 \]
 is of rank $k-1$ at $p$.
 By Proposition \ref{prop:mod},
 $\Lambda_1$ has the same rank as
 the function
 \[
 \hat \Lambda_1:=
 (\lambda_1,\tilde\zeta \lambda_1,\dots,\tilde\zeta^{k-2}\lambda_1)
 \]
 at $p$.
 Together with \eqref{eq:first1},
 we get the assertion.
\end{proof}

For the sake of simplicity, we denote $\tilde\eta\lambda$ as in 
\eqref{eq:lambda-der-1}
by $\dot\lambda$,  and
\begin{equation}\label{eq:lambda-der-dot}
   \dot\lambda:=\tilde\eta\lambda,
    \qquad
    \ddot\lambda:= \tilde\eta^{2}\lambda, \quad\dots\quad,\quad
    \lambda^{(k)}:=\tilde\eta^{k}\lambda
\end{equation}
from now on.

Let $p$ be an $\A_{k+1}$-point
of a non-degenerate homomorphism $\phi:TM^n\to \E$.
We fix an extended null vector field $\tilde \eta$
defined on a neighborhood $U$ of $p$.
Then for each $j=1$,\dots, $k-1$, it holds that
(cf.\ Definition \ref{eq:l12})
 \begin{enumerate}
  \item[(1-$j$)]\label{item:1j}
       $\lambda(p)=\dot \lambda(p)=\dots =\lambda^{(j-1)}(p)=0$,
  \item[(2-$j$)]\label{item:2j}
       and the Jacobi matrix of the $\R^j$-valued
       $C^\infty$-map 
       $\Lambda:=(\lambda,\dot \lambda,\dots,\lambda^{(j-1)})$
       is of rank $j$ at $p$.
 \end{enumerate}
By the implicit function theorem, 
there exists a neighborhood $V_j(\subset U)$ of $p$ and
an $(n-j)$-dimensional submanifold
$S^{n-j}$ such that
\begin{equation}\label{eq:S}
 S^{n-j}=\{q\in V_j\,;\,
  \lambda(q)=\dot \lambda(q)=\cdots =\lambda^{(j-1)}(q)=0
  \}.
\end{equation}
So we set $V:=\bigcap_{j=1}^{k} V_j$.

\begin{lemma}
 The restriction $\phi|_V:TV\to \E|_V$
 of $\phi$ induces the
 $j$-th non-degenerate
 reduction homomorphism
 \[
       (\phi|_V)^{(j)}:T\Sigma^{n-j}_V\longrightarrow \E^{(j)}
                     \qquad (j=1,\dots,k)
 \]
 such that the singular set $\Sigma^{n-j-1}_V$
 of $(\phi|_V)^{(j)}$ satisfies
 \begin{equation}\label{eq:ind00}
  \Sigma^{n-j-1}_V=S^{n-j-1}
 \end{equation}
 and $\lambda^{(j)}:V\to \R$ gives a $(\phi|_V)^{(j)}$-function,
 where $\Sigma^{n-j}_V$ is the singular set of
 $(\phi|_V)^{(j-1)}$.
\end{lemma}

\begin{proof}
 When $j=1$, Theorem \ref{thm:induction1} 
 implies the assertion.
 We show the assertion inductively.
 We assume that the $(j-1)$-th reduction
 $(\phi|_V)^{(j-1)}:T\Sigma^{n-j+1}_V\to \E^{(j-1)}$
 exists and the equality 
 \begin{equation}\label{eq:ind000}
    \Sigma^{n-j}_V=S^{n-j}
 \end{equation}
 holds and
 $\lambda^{(j-1)}$ is a  $(\phi|_V)^{(j-1)}$-function.
 Since 
 $p$ is an $\A_{k+1}$-point,
Theorem \ref{thm:induction2}
 yields that $p$ is an $\A_{k-j+2}$-point
 of $(\phi|_V)^{(j-1)}$.
 Since $k\ge j$, the reduction 
 \[
   (\phi|_V)^{(j)}:T\Sigma^{n-j}_V\longrightarrow \E^{(j)},
 \]
 is non-degenerate 
 if we choose a sufficiently small $V$,
 where $\Sigma^{n-j-1}_V$ is the singular set of 
 $(\phi|_V)^{(j)}$.
 Then \ref{item:ind1} of Theorem \ref{thm:induction1} implies that  
 \[
    \Sigma^{n-j-1}_V:=\{q\in \Sigma^{n-j}_V\,;\, 
          \eta_q\in T_q\Sigma^{n-j}_V\}.
 \]
 Since $\eta_q\in T_q \Sigma^{n-j}_V$ holds if and
 only if 
 \[
   \lambda^{(j)}(q)=d\lambda^{(j-1)}(\eta_q)=0,
 \]
 we have that
 \[
   \Sigma^{n-j-1}_V=\{q\in \Sigma^{n-j}_V\,;\, \eta_q
         \in T_q\Sigma^{n-j}_V\}
    =
    \{q\in \Sigma^{n-j}_V\,;\, \lambda^{(j)}(q)=0\}.
 \]
 Moreover, by \eqref{eq:ind000},
 \begin{equation}\label{eq:V}
    \Sigma^{n-j-1}_V=\{q\in \Sigma^{n-j}_V\,;\, 
           \lambda^{(j)}(q)=0\}
     =\{q\in S^{n-j}\,;\, \lambda^{(j)}(q)=0\}
     =S^{n-j-1}.
 \end{equation}
 We fix a $(\phi|_V)^{(j)}$-function
 $\lambda_j:\Sigma^{n-j}_V\to \R$.
 Since we have shown that $(\phi|_V)^{(j)}$ 
is non-degenerate,
 $d\lambda_j\ne 0$ on $\Sigma^{n-j-1}_V$.
 By \eqref{eq:V},
 the zeros of $\lambda^{(j)}$ 
 coincide with 
 those of $\lambda_j$. 
 Then the division property of
 $C^\infty$-functions yields that
 there exists a $C^\infty$-function germ $\sigma$
 on $\Sigma^{n-j}_V$
 such that
 \[
    \lambda^{(j)}=\sigma \lambda_j.
 \]
 Since $d\lambda^{(j)}(p)\ne 0$ by \ref{item:ojm:2} of Definition~\ref{fact:ojm}
 we have $\sigma(p)\ne 0$, namely $\lambda^{(j)}$
 is also a $(\phi|_V)^{(j)}$-function.
 Thus we proved the $j$-th step of the
 induction procedure.
\end{proof}

Since the singular set
of the $j$-th reduction $\phi^{(j)}$ does not depend
on  the choice of $\lambda$ and $\tilde \eta$,
we get the following assertion.

\begin{proposition}\label{lemma:inv}
 Let $p$ be an $\A_{k+1}$-point of
 a non-degenerate homomorphism
 $\phi:TM^n\to \E$, and
 $\tilde \eta$ an extended null vector field
 defined on a neighborhood $U$ of $p$.
 Then there exists a neighborhood $V(\subset U)$ of $p$
 such that
 \begin{align*}
  \Sigma^{n-j}_V:
  &=\{q\in V\,;\,
  \lambda(q)=\cdots=\lambda^{(j-1)}(q)=0\} \\
  &=
  \{q\in \Sigma^{n-j+1}_V\,;\,
  \tilde \eta_q\in T_q\Sigma^{n-j+1}_V\}
  \qquad (j=1,\dots,k)
 \end{align*}
 is an $(n-j)$-dimensional submanifold of $V$.
 Moreover, each $\Sigma^{n-j}_V$ does not
 depend on the choice of $\lambda$ and $\tilde \eta$.
 Furthermore, the following equalities hold
 \[
       \AA_{2}\cap V=\Sigma^{n-1}_V
               \setminus \Sigma^{n-2}_V,\dots,
                \AA_{k}\cap V=
              \Sigma^{n-k+1}_V
                   \setminus \Sigma^{n-k}_V,\quad
                    \AA_{k+1}\cap V
                =\Sigma^{n-k}_V,
 \]
 where $\AA_{j+1}$ $(j=1,\dots,k)$ is the set 
 of $A_k$-points
 of $\phi$.
\end{proposition}

In this paper, we mainly discuss on
bundle homomorphisms having only 
$A_{k+1}$-singularities ($1\le k\le n$), 
so we give the following definition.

\begin{definition}\label{def:Morin}
 A non-degenerate homomorphism $\phi:TM^n\to \E$
 is called a {\it Morin homomorphism\/} if 
 the set of singular points of $\phi$ consists of
 $\A_k$-points for $k=2,3,\dots,n+1$.
 A Morin homomorphism $\phi$ is called of
 {\it depth $k$\/} if 
 $\A_{k+1}$-points exist
 but there are no 
 $\A_{k+2}$-points on $M^n$.
\end{definition}

The following assertion follows 
immediately from the definition of Morin homomorphisms.

\begin{proposition}
 Let $\phi:TM^n\to \E$ be a
 non-degenerate homomorphism and $p\in M^n$ 
 an $A_{k+1}$-point. Then
 there exists a neighborhood $U$ of $p$
 such that the restriction of $\phi$
 into $U$ gives a Morin
 homomorphism.
\end{proposition}

\begin{proof}
 Take an extended null vector field $\tilde \eta$
 defined on $U$.
 Since $p$ is an $A_{k+1}$-point,
 there exists a neighborhood $U$ of $p$
 such that 
 \begin{itemize}
  \item $\lambda^{(k)}\ne 0$ on $U$, and
  \item the Jacobi matrix of $\Lambda$ as in
	Definition \ref{fact:ojm} is of rank $k$
	on $U$,
 \end{itemize}
 where $\lambda$ is a local $\phi$-function
 defined on $U$. 
 Let $q\in U$ be a singular point of $\phi$.
 Then there exists a positive 
 integer $j(\le k)$ such that 
 \[
    \lambda^{(0)}(q)=\cdots =\lambda^{(j-1)}(q)=0,\qquad
    \lambda^{(j)}(q)\ne 0.
 \]
 Then $q$ is an $A_{j+1}$-point, proving the
 assertion.
\end{proof}

Moreover, as a corollary of 
Theorem \ref{thm:induction2},
we get the following assertion.

\begin{proposition}\label{prop:redMorin}
 Let $\phi:TM^n\to \E$ be a Morin homomorphism
 of depth $k(\ge 2)$.
 Then its reduction
 $\hat \phi:T\Sigma^{n-1}\to \hat \E$ is
 a Morin homomorphism
 of depth $k-1$. 
\end{proposition}

Suppose that $\phi:TM^n\to \E$
is a Morin homomorphism of depth $k$.
By Proposition \ref{lemma:inv},
\[
   \Sigma^{n-j}:=\{p\in M^n\,;\,
                 \lambda(p)=\cdots=\lambda^{(j-1)}(p)=0\}
              \qquad (j=1,\dots,k)
\]
does not depend on the choice of a $\phi$-function
$\lambda$ and the extended null vector field $\tilde \eta$,
that is, it is well-defined as an $(n-j)$-dimensional
submanifold of $M^n$, and
\[
   \AA_{2}=\Sigma^{n-1}
         \setminus \Sigma^{n-2}\,\,,\dots,\,\,
                \AA_{k}=
      \Sigma^{n-k+1}
           \setminus \Sigma^{n-k},\quad
             \AA_{k+1}
         =\Sigma^{n-k}.
\]
In this case, we give the following conventions
\[
    \AA_{k+j+1}=
               \Sigma^{n-k-j}=\emptyset \qquad (1\le j\le n-k).
\]

We now consider the case that $\E$ is orientable.
Then, there is a non-vanishing section $\mu$
of the determinant line bundle of the dual 
bundle $\E^*$ of $\E$
defined on $M^n$.
We call $\mu$ an  {\it orientation\/} of $\E$.
In this case,
there is a unique $C^\infty$-function 
$\lambda:M^n\to \R$
such that
\begin{equation}\label{eq:lambda2}
    \phi^*\mu=\lambda\, \Omega,
\end{equation}
where $\Omega$ is
an {\it orientation\/} of $M^n$.
We call $\lambda$ the {\it $\phi$-function\/}
associated to $\mu$ and $\Omega$
defined on $M^n$.
We set 
\begin{equation}\label{eq:pm}
  M^n_+:=\{p\in M^n\,;\, \lambda(p)>0\},\qquad
  M^n_-:=\{p\in M^n\,;\, \lambda(p)<0\}.
\end{equation}
Then
$\Sigma^{n-1}$ 
coincides with the boundary 
$\partial M^{n}_+=\partial M^{n}_-$.

\begin{definition}
 Let $\phi:TM^{n}\to \E$ be a non-degenerate
 bundle homomorphism and $\lambda$ a 
 $\phi$-function associated to $\mu$ and $\Omega$.
 A $\phi$-function $\tau:U\to \R$ 
 defined on an open subset $U(\subset M^n)$
 is called an {\it oriented $\phi$-function\/}
 if there exists a positive valued function
 $\sigma\in C^\infty(U)$ such that
 $\tau=\sigma \lambda$ on $U$.
\end{definition}

Our definition of Morin homomorphisms
is motivated by the existence of the
following two typical examples:
Let $m$, $n$ be two positive integers.
Two differentiable map germs 
$f_i\colon{}(\R^m,p_i)\to(\R^{n},q_i)$ $(i=1,2)$ are
{\em right-left equivalent\/} if there exist diffeomorphism
germs 
$\psi\colon{}(\R^m,p_1)\to (\R^m,p_2)$ and 
$\Psi:(\R^{n},q_1)\to(\R^{n},q_2)$
such that $\Psi\circ f_1=f_2\circ\psi$.

\begin{definition}\label{def:Morin0}
 The {\it Morin-$k$-singularities\/}
 ($1\le k \le n$)
 are map germs which 
 are right-left equivalent to
 \[
   f(x_1,\dots,x_n) =\biggl(
          x_1x_n +x_2 (x_n)^2+\cdots+ x_{k-1}(x_n)^{k-1}
         + (x_n)^{k+1}
               ,x_1,\dots,x_{n-1}\biggr)
 \]
 at the origin. 
 The Morin-$0$-singularities 
 mean regular points.
\end{definition}

\begin{example}\label{ex:Morin}
 Let $M^n$ and $N^n$ be oriented $n$-manifolds, 
 and let 
 $f:M^n\to N^n$ be a $C^\infty$-map having only 
 Morin singularities.
 Then the differential $df$
 of $f$ canonically induces a Morin homomorphism
 (cf.\ Appendix of \cite{SUY3})
 \[
    \phi=df:TM^n\longrightarrow \E_f:=f^*TN^n.
 \]
 Let $\omega_{M^n}$ and $\mu_{N^n}$ be the fundamental
 $n$-forms of $M^n$ and $N^n$, respectively.
 Then there exists a $C^\infty$-function
 $\lambda$ on $M^n$ such that
 $f^*\mu_{N^n}=\lambda \omega_{M^n}$, 
 which gives an oriented $\phi$-function.
 The set $M^n_+$ (resp.\ $M^n_-$) coincides 
 with the set where $\lambda>0$ (resp.\ $\lambda<0$).
 The sign of  $\lambda$ coincides with the sign
 of the Jacobian of $f$  with respect to oriented 
 local coordinate  systems of $M^n$ and $N^n$.
 In this case, Morin-$k$-points of the map $f$
 are $\A_{k+1}$-points of the homomorphism $\varphi=df$
 (see \cite[Theorem A1]{SUY3}).
 When $(N^n,ds^2)$ is a Riemannian manifold,
 then the pull-back bundle $f^*TN^n$ on $M^n$
 has a canonical coherent tangent bundle structure
 (cf.\ \cite{SUY6}).  
\end{example}

\begin{definition}\label{def:front}
 The $A_{k+1}$-type  singularity 
 (or $A_{k+1}$-front singularity)
 is a map germ defined by
 \begin{equation}\label{eq:ak-def}
    X
      \longmapsto
    \left(
        (k+1)t^{k+2}+\sum_{j=2}^k(j-1)t^jx_j,
        -(k+2)t^{k+1}-\sum_{j=2}^kjt^{j-1}x_j,
        X_{1}
    \right)
 \end{equation}
 at the origin, where $X=(t,x_2,\dots,x_n)$ and $X_{1}=(x_2,\dots,x_n)$.
 Its image coincides 
 with the discriminant set 
 $\{F=F_t=0\}\subset (\R^{n+1};u_0,\dots,u_n)$ 
 of
 the versal unfolding 
 \begin{equation}\label{eq:versal}
    F(t,u_0,\dots,u_n):=t^{k+2}+u_{k} t^k+\dots +u_1 t +u_0.
 \end{equation}
 By definition, $A_1$-front singularities are regular points.
 A $3/2$-cusp in a plane is an $A_2$-front singularity
 and a swallowtail in $\R^3$ is an $A_3$-front singularity.
\end{definition}

\begin{example}\label{ex:wave}
 Let $f:M^n\to \R^{n+1}$ be a wave front
 which admits only $\A_{k+1}$-type singularities
 ($k=1,\dots,n$). 
 Suppose that $f$ is {\it co-orientable\/},
 that is, there exists a globally
 defined unit normal vector field $\nu$ along $f$.
 Let $f^*T\R^{n+1}$
 be the pull-back of $T\R^{n+1}$ by $f$, and
 consider the subbundle $\E_f$ of $f^*T\R^{n+1}$ 
 whose fiber $\E_p$ at $p\in M^n$ is 
 the orthogonal complement of $\nu_p$.
 Then the differential $df$ of $f$ induces
 a bundle homomorphism
 \[
     \phi_f=df:TM^n\ni v \longmapsto df(v)\in \E_f
 \]
 called the {\it first homomorphism\/} of $f$
 as in \cite[Section 2]{SUY4},
 which gives a Morin homomorphism
 (cf.\ Appendix of \cite{SUY3}).
 Consider the function
 \[
     \lambda:=\det(f_{x_1},\cdots,f_{x_{n}},\nu),
 \]
 where $f_{x_i}:=\partial f/\partial x_i$ ($i=1,\dots,n$)
 and $(x_1,\dots,x_n)$ is an oriented 
 local coordinate system of $M^n$.
 Then $\lambda$ is an oriented $\phi$-function of $\E_f$,
 and the set $M^n_+$ (resp.\ $M^n_-$) coincides 
 with the set where $\lambda>0$ (resp.\ $\lambda<0$).
 Moreover, $A_{k+1}$-front singular points  of the map $f$
 are $\A_{k+1}$-points of the homomorphism $\varphi=df$
 (see \cite[Corollary 2.5]{SUY3}).
 As in the case of the previous example,
 $\E_f$ has a canonical coherent tangent 
 bundle structure (cf.\ \cite{SUY6}).
\end{example}

\begin{remark}\label{rmk:Morin}
 As seen in Examples \ref{ex:Morin}
 and \ref{ex:wave}, our definition of $\A_k$-points 
 gives a unified intrinsic treatment of singularities
 of both Morin maps of the same dimension
 and the $A_k$-singularities appearing in 
 hypersurfaces in $\R^{n+1}$.
 In this intrinsic treatment, 
 the usual $k$-th singular points
 for Morin maps and the
 $A_{k+1}$-points for wave fronts
 are both regarded as $\A_{k+1}$-points of 
 bundle homomorphisms.
 In other words, 
 the order of singularities of Morin maps is
 not synchronized with
 the order of singularities of the corresponding 
 bundle homomorphisms.
 For example, a fold (i.e.\ a Morin-1-singularity)
 and a cusp (i.e.\ a  Morin-2-singularity)
 induce  an $\A_2$-point and
 an $\A_3$-point of bundle homomorphism,
 respectively.
\end{remark}

\section{Characteristic vector fields}\label{sec:char}
We fix a Morin homomorphism $\phi:TM^n\to \E$, where
$M^n$ is an oriented compact $n$-manifold.
We now suppose that $\E$ is oriented,
and fix an oriented $\phi$-function $\lambda:M^n\to \R$.
Then the singular set  $\Sigma^{n-j}$ ($j=0,\dots,n$)
of the $(j-1)$-th reduction $\phi^{(j-1)}$
defined in the previous section 
is an orientable submanifold of $M^n$,
unless it is empty.

\begin{proposition}\label{prop:k_even}
 If $k$ $(2\le k\le n)$ is even, then 
 the sign of the function $\lambda^{(k)}$ 
 does not depend on the choice of 
 the extended null vector field $\tilde \eta$.
\end{proposition}

\begin{proof}
 Even if we change the extended
 null vector field $\tilde \eta$ to $-\tilde\eta$,
 the sign of the 
 function $\lambda^{(k)}$ on the set $\Sigma^{n-k}$
 does not change, since $k$ is even.
\end{proof}

Hence, for each even integer $k$
($2\le k \le n$),
we can set
\[
  \Sigma^{n-k}_+:=\{p\in \Sigma^{n-k}\,;\, 
                          \lambda^{(k)}(p)>0\}, \quad
  \Sigma^{n-k}_-:=\{p\in \Sigma^{n-k}\,;\, 
                          \lambda^{(k)}(p)<0\}.
\]
As a convention, we define
$\Sigma^{n}_+=M^n_{+}$ and
$\Sigma^{n}_-=M^n_{-}$,
where $M^n_{\pm}$ are as in \eqref{eq:pm}.
Also, the following assertion holds:
\begin{proposition}\label{prop:k_odd}
 Let $k$ be an odd  positive integer, and
 $p$ an $\A_{k+1}$-point.
 Then the scalar multiple $\lambda^{(k)}\eta$
 of the null vector field $\eta$  
 along $\Sigma^{n-k}$
 points toward the domain $\Sigma^{n-k+1}_+$
 at $p$, where $\Sigma^n:=M^n$.
\end{proposition}

\begin{proof}
 We now take a Riemannian metric $ds^2$ on $M^n$.
 We denote by $ds^2_{n-k+1}$ the Riemannian metric
 of $\Sigma^{n-k+1}$ induced by $ds^2$. 
 Then the hypersurface $\Sigma^{n-k}$ 
 embedded in $\Sigma^{n-k+1}$
 can be characterized as the level set $\lambda^{(k-1)}=0$.
 Then we have that
 \[
    ds^2_{n-k+1}\bigl(\tilde \eta_p,
       \grad(\lambda^{(k-1)})_p\bigr)
      =d\lambda^{(k-1)}_p(\tilde \eta_p)=\lambda^{(k)}(p),
 \]
 where ``$\grad$'' denotes the gradient of the function
 with respect to the metric $ds^2_{n-k+1}$.
 Thus
 $ds^2_{n-k+1}\bigl(\lambda^{(k)}\tilde \eta,\grad(\lambda^{(k-1)})\bigr)$
 is positive at $p$.
 Since $\grad(\lambda^{(k-1)})$
 gives  a normal vector field
 along $\Sigma^{n-k}$ 
 pointing toward $\Sigma^{n-k+1}_+$,
 the assertion is proven.
\end{proof}

\begin{definition}\label{def:pm-ak}
 Let $\phi:TM^n\to \E$ be a Morin 
 homomorphism  and $p$ an $\A_{2k+1}$-point.
 Since the sign of  $\lambda^{(2k)}(p)$
 does not depend on the $\pm$-ambiguity of 
 the choice of extended 
 null vector field $\tilde \eta$
 (cf.\ Proposition \ref{prop:k_even}),
 we call $p$ 
 a {\it positive $\A_{2k+1}$-point\/} 
 (resp.\ a {\it negative $\A_{2k+1}$-point})
 if $\lambda^{(2k)}(p)$ is positive (resp.\ negative).
\end{definition}
The set of positive (resp.\ negative) 
$\A_{2k+1}$-points  is denoted by 
$\AA^+_{2k+1}$ (resp.\ $\AA^-_{2k+1}$).
Then the equalities
\begin{equation}\label{eq:pm-ak}
 \begin{aligned}
   \AA^+_{2k+1}&
      :=\{p\in \AA_{2k+1}\,;\,\lambda^{(2k)}(p)>0\}
       =\Sigma^{n-2k}_+\setminus \Sigma^{n-2k-1}, \\
   \AA^-_{2k+1}&
      :=\{p\in \AA_{2k+1}\,;\,\lambda^{(2k)}(p)<0\}
       =\Sigma^{n-2k}_-\setminus \Sigma^{n-2k-1} 
  \end{aligned}
\end{equation}
hold. 
If $n=2$ and $f:M^2\to \R^3$ 
is a wave front, then positive
(resp.\ negative) $\A_3$-points 
as in Example \ref{ex:wave} correspond to 
positive (resp.\ negative) swallowtails.

Let $X$ be a vector field of $M^n$ which vanishes 
at $p\in M^n$.
Take a local coordinate system $(U;x_1,\dots,x_n)$
at $p$ and write
\[
   X=\xi_1\frac{\partial}{\partial x_1}
             +\dots +\xi_n\frac{\partial}{\partial x_n}.
\]
Then a zero $p$ of $X$ is called {\it generic\/}
if the Jacobian of the map
\[
   U\ni q \longmapsto 
\bigl(\xi_1(q),\dots, \xi_n(q)\bigr)\in \R^n
\]
does not vanish at $q=p$. 
A vector field $X$ defined on $M^n$ is called {\it generic\/}
if all its zeros are generic.

\begin{definition}\label{def:X}
 Let
 $\varphi\colon{}TM^n\to\E$
 be a
 Morin homomorphism of depth $k$
 ($k=1,\dots,n$).
 A $C^\infty$-vector field $X$ defined on $M^n$ is called
 a {\it characteristic vector field\/} 
 of $\varphi$ 
 if it satisfies the following three conditions.
\begingroup
 \renewcommand{\theenumi}{(\roman{enumi})}
 \renewcommand{\labelenumi}{(\roman{enumi})}
 \begin{enumerate}
  \item\label{char:0}
       $X$ is a generic vector field on $M^n$
       which does not vanish at any point of $\Sigma^{n-1}$.
  \item\label{char:1}
       For each $j=n-k,\dots,n-1$, there exists a
       generic tangent vector field $X_j$ of 
       $\Sigma^{j}$ such that the equality $\phi(X)=\phi(X_j)$ 
       holds on $\Sigma^{j}$ and $X_j$ has no zeros on $\Sigma^{j-1}$.
\item\label{char:2}
     For each $\A_{l+1}$-point $p$
     ($l=1,\dots,k$)
     (namely, $p\in \Sigma^{n-l}\setminus \Sigma^{n-l-1}$)
     satisfying $\phi(X_p)$=0, 
     there exists a neighborhood
     $U$ of $p$ of $M^n$
     such that the restriction of $X$ to
     $U\cap \Sigma^{n-l+1}$ coincides with 
     $X_{n-l+1}$
     on $U\cap \Sigma^{n-l+1}$ (cf.\ Figure \ref{fig:zu}).
     Moreover, if $l$ is odd,
     $X$ points into $\Sigma^{n-l+1}_+$
     at $p\in \Sigma^{n-l}$.
\end{enumerate}
\endgroup
\end{definition}

\begin{remark}\label{rmk:def}
 Let $X$ be a characteristic vector field 
 on $M^n$. 
 If $k=n$, 
 then $\phi(X)$ must vanish at 
 each $\A_{n+1}$-point.
(In fact, since any null vector fields
 are tangent to $\Sigma^1$ at each 
 $\A_{n+1}$-point $p$,
 the property \ref{char:1} yields that
 $X_1$ points in the null direction  
 at $p$, and $X=X_1$ near $p$ on $\Sigma^1$
 by \ref{char:2}.) 
\end{remark}

In this section, we shall construct a
characteristic vector field, which will play
a crucial role in proving formula \eqref{eq:main} in 
the introduction:

\begin{proposition}\label{prop:charcaristic}
 Let $M^n$ be a compact oriented manifold, 
 and $\phi:TM^n\to \E$
 a Morin homomorphism.
 Suppose that $\E$ is oriented.
 Then, there exists a characteristic vector field defined
 on $M^n$ associated to $\phi$.
\end{proposition}

To prove the assertion, 
we prepare the following:

\begin{lemma}\label{lem:new-da}
 Let $M^n$ be a compact orientable
 manifold, and 
 $\varphi\colon{}TM^n\to\E$
 $(n\ge 1)$ 
 a Morin homomorphism of depth $k$
 $(k\ge 1)$ and $X$ a generic vector field
 on $\Sigma^{n-1}$ such that it does not have
 any zero on a compact subset $C(\subset \Sigma^{n-1})$.  
 {\rm(}Here we are not assuming that $\E$ is orientable.{\rm)}
 Then there exists a vector field $\tilde X$
 satisfying the following properties:
 \begin{enumerate}
  \item\label{char:0r}
       $\tilde X$ is a generic vector field on $M^n$
       which has no zeros on $\Sigma^{n-1}$.
  \item\label{char:1r}
       $\phi(\tilde X)=\phi(X)$ holds on $\Sigma^{n-1}$.
  \item\label{char:2r}
       $\tilde X=X$ on $C$. 
 \end{enumerate}
\end{lemma}

\begin{proof}
 We fix a Riemannian metric on $M^n$.
 Since $M^n$ is orientable,
 we can take $\norm$
 as a normal vector field
 defined on $\Sigma^{n-1}$.
 Taking $\delta$ to be sufficiently small,
 there exists a canonical diffeomorphism
 \[
    \exp:\Sigma^{n-1}\times [-\delta,\delta]
      \longrightarrow \overline{\n_\delta(\Sigma^{n-1})}
 \]
 such that $t\mapsto \exp(q,t)$ is the normal 
 geodesic of $M^n$ with arclength parameter
 starting from each $q\in \Sigma^{n-1}$ 
 in the direction
 $\norm$. 
 Here
 $\n_\delta(\Sigma^{n-1})$
 is the $\delta$-tubular neighborhood
 of $\Sigma^{n-1}$ in $M^n$.
 Then
 \[
    \tilde{\norm}(q,s):=\frac{\partial \exp(q,s)}{\partial s}
 \]
 gives a unit vector field defined on 
 $\n_\delta(\Sigma^{n-1})$
 as an extension of $\norm$.
 Take an open neighborhood $V$ of $C$
 as an open subset of $\Sigma^{n-1}$
 such that the closure $\overline V$ 
 of $V$ is compact
 and $X$ has no zeros on $\overline V$.
 Without loss of generality, we may assume that
 the normal vector $\norm$ is proportional to
 the null vector field on $\Sigma^{n-1}\setminus V$.
 Let 
 $\rho:\Sigma^{n-1} \to [0,1]$
 be a smooth function such that
 \[
   \rho(q)=
     \begin{cases}
      1 & (\mbox{if $q\in C$}), \\
      0 & (\mbox{if $q\not\in V$}). 
     \end{cases}
 \]
 Let $W$ be the vector field on
 $\n_\delta(\Sigma^{n-1})$ obtained via parallel transport of
 $X$ along each normal geodesic $s\mapsto \exp(q,s)$.
 We set
 \begin{equation}
   \tilde W(q,s):=W(q,s)+ 
    \biggl(s^2\rho(q)+\left(1-\rho(q)\right)\biggr)
    \tilde{\norm}(q,s),
 \end{equation}
 which is a vector field on $\n_\delta(\Sigma^{n-1})$.
 Then $\tilde W$ has no zeros on $\n_\delta(\Sigma^{n-1})$
 since $X$ has no zeros on $\n_\delta(\Sigma^{n-1})$.
 We then apply Lemma~\ref{lem:app} in the appendix
 by setting $K=\overline{\n_{\delta}(\Sigma^{n-1})}$ 
 and get a generic vector field $\tilde X$
 defined on $M^n$ such that $\phi(\tilde X)$
 coincides with $\phi(X)$ on $\Sigma^{n-1}$.
 It can be easily checked that $\tilde X$ is the 
 desired vector field.
\end{proof}

\begin{figure}[htb]
 \begin{center}
 \includegraphics[height=3cm]{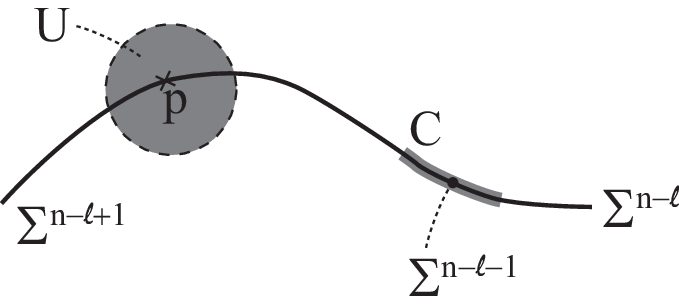}  
\caption{Proof of Proposition \ref{prop:charcaristic}}
\label{fig:zu}
 \end{center}
\end{figure}

\begin{proof}[Proof of Proposition \ref{prop:charcaristic}]
 We prove the assertion by induction of the 
 depth $k$ of the Morin homomorphism.
 So we firstly consider the case that $k=1$.
 Suppose that $n\ge 2$.
 Then $\Sigma^{n-1}$ is positive dimensional.
 We take a generic vector field $X$  on $\Sigma^{n-1}$
 and apply the
 previous lemma by setting $C$ to be
 the empty set.

 Next we consider the case that $n=1$.
 Let $p_1$,\dots, $p_m$ be $A_2$-points on $M^1$.
 Then we can take an extended null vector 
 field $\tilde \eta_j$ defined on a neighborhood $U_j$
 of $p_j$ which has no zeros on $\overline{V_j}(\subset U_j)$.
 We may assume that 
 the $V_j$'s are pairwise disjoint.
 Applying Lemma \ref{lem:app} by 
 setting $K=\overline{V_1}\cup \cdots \cup \overline{V_m}$, 
 we can get a generic  vector field $X$ on $M$ such that
 $X=\tilde \eta_j$
 on $V_j$ for $j=1$, $2$,\dots, $m$, 
 which gives the properties \ref{char:0}--\ref{char:2}.

 We now assume that the assertion holds for $k-1$.
 We fix an inner product  $\inner{~}{~}$
 on $\E$. 
 As in the proof of Theorem \ref{thm:induction1},
 we can take a unit section $\vect{u}$
 such that the induced bundle $\hat\E$ 
 defined by \eqref{eq:quotient}
 is the subbundle of $\E$ which is orthogonal
 to $\vect{u}$.
 Using the assumption of induction,
 there exists a vector field $X$ satisfying the 
 properties \ref{char:0}--\ref{char:2}
 on $\Sigma^{n-1}$ for $\hat \E$
 by taking $\dot \lambda$ to be a 
 $\hat \phi$-function.
 Let $\delta$ be a small positive number
 such that $X$ has no zeros on
 $C:=\overline{\N_\delta(\Sigma^{n-2})}$,
 where $\N_\delta(\Sigma^{n-2})$
 is a $\delta$-tubular neighborhood
 of $\Sigma^{n-2}$ in the Riemannian
 manifold $\Sigma^{n-1}$.
 We apply Lemma \ref{lem:new-da} for $X$
 (see Figure \ref{fig:zu}),
 and we get the vector field $\tilde X$
 satisfying the properties \ref{char:0r}--\ref{char:2r}.
 Then $\tilde X$ satisfies \ref{char:0}, \ref{char:1} and \ref{char:2}
 by construction.
 The property \ref{char:1}
 follows from \ref{char:1r}.
\end{proof}

\section{Adapted coordinate systems and the two dimensional case}
\label{sec:even2}

\begin{proposition}\label{thm:coord}
 Let $\phi:TM^n\to \E$ be a Morin homomorphism
 on an $n$-manifold $M^n$.
 Then there exists a local coordinate 
 system $(U;x_1,\dots,x_n)$   
 centered at an $\A_{k+1}$-point $p\in M^n$  $(k\ge 1)$
 satisfying the following properties{\rm:}
 \begin{enumerate}
  \item\label{coord:2}
       For each $j=1$,\dots, $k-1$, 
       the restriction of 
       $\{\partial/ \partial x_{j+1},
               \dots,\partial/ \partial x_{n}\}$
       spans the tangent space of $\Sigma^{n-j}$ at $p$,
  \item\label{coord:3} 
       $\partial/\partial x_{k}$ gives an
       extended null vector field 
       on $U$.
 \end{enumerate}
\end{proposition}

The local coordinate system $(x_1,\dots,x_n)$ given 
in Proposition \ref{thm:coord}
is called a {\it $\phi$-adapted coordinate system\/} at $p$.

\begin{proof}
 Let $\lambda:U\to \R$ be a $\phi$-function 
 defined on a local coordinate neighborhood 
 $(U;y_1,\dots,y_n)$ of $p$.
 Let $\tilde \eta$ be an extended null vector field
 on $U$ and $\eta$ its restriction 
 to $\Sigma^{n-1}\cap U$.
 Then by \ref{item:ojm:2} of Definition \ref{fact:ojm},
 we have that
 \[
  \frac{\partial(\lambda,\dot \lambda,\dots,\lambda^{(k-1)})}
       {\partial(y_1,\dots,y_k)}\ne 0.
 \]
 By the implicit function theorem,
 there exist functions
 $y_j(y_{k+1},\dots,y_n)$ ($j=1,\dots,k$)
 such that
 $y_j(0,\dots,0)=0$ and
 \[
       \lambda^{(j-1)}(y_1(\hat y),\dots,y_k(\hat y),\hat y)=0
             \qquad (j=1,\dots,k),
 \]
 where 
 $\hat y=(y_{k+1},\dots,y_n)$ and $\lambda^{(0)}:=\lambda$.
 So if we set
 \[
       x_1:=\lambda,\,\, x_2:=\dot \lambda,\,\,
           \dots,\,\, x_k:=\lambda^{(k-1)}, \quad x_l:=y_l 
             \qquad (l=k+1,\dots,n),
 \]
 then $\psi:=(x_1,\dots,x_n)$ gives a new
 local coordinate system at $p$ satisfying the property 
 \ref{coord:2}.
 Then the restriction $\eta|_{\Sigma^{n-k}}$
 of the null vector field 
 is a tangent vector field of $\Sigma^{n-k+1}$
 along $\Sigma^{n-k}$,
 and can be written as
 \[
      \eta|_{\Sigma^{n-k+1}}=\sum_{j=k}^n c_j \partial_j,
 \]
 where $\partial_{j}:=\partial/\partial x_j$
 ($j=k,\dots,n$).
 Since $\eta$ is transversal to $\Sigma^{n-k}$ at $p$,
 the coefficient $c_k$ does not vanish.
 Let $\{g_t\}_{|t|<\epsilon}:V\to M^n$
 be the local $1$-parameter group of transformations
 generated by $\tilde \eta$, where $V(\subset U)$ 
 is a 
 neighborhood of $p$ in $M^n$ and $\epsilon>0$
 is a small positive number.
 Then 
 \[
    \Phi:(t_1,t_2,\dots,t_{k-1},t_k,t_{k+1},\dots,t_n)
      \mapsto g_{t_k}(\psi(t_1,t_2,\dots,t_{k-1},0,t_{k+1},\dots,t_n))
 \]
 gives a local diffeomorphism
 such that  the equalities
 \[
    d\Phi(\partial/\partial t_k)=\eta,\quad
          d\Phi(\partial/\partial t_l)=\partial/\partial x_l
           \qquad (l=k+1,\dots,n)
 \]
 hold, and they span
 the tangent space of $\Sigma^{n-k+1}$
 when $t_1=t_2=\cdots=t_{k}=0$.
 Thus the inverse map 
 $\Phi^{-1}$ gives the desired local coordinate system.
\end{proof}

Here we prove formula \eqref{eq:main} for $n=2$.
Although this formula was proved as a corollary of the
Gauss-Bonnet type formula
in \cite{SUY5} and \cite{SUY6}, our proof in
this section is new.

Let $X$ be a characteristic vector field associated to 
a Morin homomorphism $\phi:TM^{2}\to \E$ of depth 
at most $2$
on a compact oriented $2$-manifold,
and we assume that $\E$ is oriented.
Take a section $Y$ of $\E$ as $Y:=\phi(X)$.
Then the following assertion holds:
\begin{proposition}\label{prop:dim2_A}
 Let $Z(Y)$, $Z(X)$ be the set of zeros on $M^2$ of $Y$ and $X$,
 respectively, and let $Z(X_1)$  be the zeros on $\Sigma^1$
 of $X_1$ {\rm(}as in Definition \ref{def:X}{\rm)}.
 Then it holds that
 \begin{align}
  &Z(Y)\cap (M^2\setminus \Sigma^1)=Z(X), \label{eq:2d1} \\
  &Z(Y)\cap (\Sigma^1\setminus \Sigma^0)=Z(X_1)\subset \AA_2, 
         \label{eq:2d2}\\
  &Z(Y)\cap \Sigma^0=\AA_3.
  \label{eq:2d3}
 \end{align}
\end{proposition}

\begin{proof}
 Since $Y=\phi(X)$,  property \ref{char:0}
 in Definition~\ref{def:X}
 implies that $Z(X)\subset Z(Y)$.
 Since $\phi:T_{p}M^2\to \E_{p}$ is a linear isomorphism 
 when $p\in M^2\setminus \Sigma^1$,
 we have \eqref{eq:2d1}.
 Since $Z(X_1)\cap \Sigma^0$
 is the empty set,
 property \ref{char:1}
 of characteristic vector field
 yields
 \[
     Z(Y)\cap (\Sigma^1\setminus \Sigma^0)= Z(X_1).
 \]
 Since $Y=\phi(X_1)$ on $\Sigma^1$
 and $X_1$ is proportional to a 
 null vector  at each $\A_3$-point $p$, we obtain 
\eqref{eq:2d3}.
\end{proof}

When $n=2$, \eqref{eq:Y} reduces to
\begin{equation}\label{eq:dim2}
  \chi^{}_\E=
     \sum_{p\in M^2\setminus \Sigma^1}\ind_p(Y)+
     \sum_{p\in \AA_2}\ind_p(Y)+
     \sum_{p\in \AA_3}\ind_p(Y).
\end{equation}

\begin{proposition}\label{prop:dim2_B}
 The first term of the right-hand 
 side of \eqref{eq:dim2} satisfies
 \begin{equation}\label{eq:1-st}
  \sum_{p\in M^2\setminus \Sigma^1}\ind_p(Y)
   =\chi(M^2_+)-\chi(M^2_-).
 \end{equation}
\end{proposition}
\begin{proof}
 Let $p$ be in $Z(Y)\setminus\Sigma^1$, 
 and $\lambda$ be an oriented $\phi$-function on 
 a neighborhood of $p$.
 We denote by $\sgn(\lambda(p))$ the sign of the
 function $\lambda$ at the
 point $p$.
 Since $\sgn(\lambda(p))=1$ (resp.\ $\sgn(\lambda(p))=-1$)
 if $\phi_p:T_pM^2\to \E_p$ is orientation preserving 
 (resp.\ orientation reversing),  we have that 
 \[
   \ind_p(Y)=
    \sgn\bigl(\lambda(p)\bigr)\ind_p(X)\qquad 
           (p\in M^2\setminus \Sigma^1).
 \]
 We set
 \[
   \bar M^2_+(\delta):=
            \overline{M^2_+\setminus \N_\delta(\Sigma^1)},
   \quad
   \bar M^2_-(\delta):=\overline{M^2_-\setminus \N_\delta(\Sigma^1)}
   \qquad (\delta>0),
 \]
 where $\N_{\delta}(\Sigma^1)$ 
 is the $\delta$-tubular neighborhood of $\Sigma^1$ as in 
 the proof of Lemma~\ref{lem:new-da}, and 
 the overline means the closure operation. 
 If we choose $\delta$ sufficiently small, then
 $Z(Y)\cap (M^2\setminus \Sigma^1)$ is contained
 in $\bar M^2_+(\delta)\cup \bar M^2_-(\delta)$
 and $\bar M^2_+(\delta)$ (resp.\ $\bar M^2_-(\delta)$)
 has the same homotopy type
 as $M^2_+$ (resp.\ $M^2_-$).
 In particular, the following identity holds
 \begin{equation}\label{eq:YX}
     \sum_{p\in M^2\setminus \Sigma^1} \ind_p(Y)=
           \sum_{p\in \bar M^2_+(\delta)} \ind_p(X)-
	   \sum_{p\in \bar M^2_-(\delta)} \ind_p(X).
 \end{equation}
 Here, $-X$ (resp.\ $X$) is an outward vector 
 of $\bar M^2_+(\delta)$ (resp.\ $\bar M^2_-(\delta)$)
 by property \ref{char:2} of Definition~\ref{def:X}
 of the characteristic vector field $X$.
 Since the operation $X\mapsto -X$ is orientation preserving,
 applying the Poincar\'e-Hopf index formula (cf.\ \cite{Mi}),
 we have that
 \[
   \chi(M^2_+)= \chi(\bar M^2_+(\delta))
      =\sum_{p\in \bar M^2_+(\delta)} \ind_p(-X)
      =\sum_{p\in \bar M^2_+(\delta)} \ind_p(X)
               =\sum_{p\in M^2_+} \ind_p(X).
 \]
 Similarly, we can also show that 
 \[
   \chi(M^2_-)=\sum_{p\in M^2_-} \ind_p(X),
 \]
 which proves the assertion.
\end{proof}
\begin{proposition}\label{prop:dim2_add}
 The second term of the right-hand 
 side of \eqref{eq:dim2} satisfies
 \begin{equation}\label{eq:1-stadd}
  \sum_{p\in \AA_2}\ind_p(Y)   =0.
 \end{equation}
\end{proposition}

\begin{proof}
 We fix $p$ in $Z(Y)\cap (\Sigma^1\setminus\Sigma^0)$.
 Then $p$ is an $\A_2$-point.
 Let $(U;x_1,x_2)$ be a
 $\varphi$-adapted coordinate system as in Proposition~\ref{thm:coord}
 (for $n=2$ and $k=2$)
 around $p$ which is compatible with
 the orientation of $M^2$.
 Then by \ref{coord:3} of Proposition~\ref{thm:coord},
 \[
    \tilde \eta:=\partial/\partial x_1
 \]
 gives an extended null vector field on $U$.
 Let $\mu$ be an orientation (i.e. a non-vanishing section
 of the determinant bundle of $\E^*$ which is 
 compatible with the
 orientation
 of $\E$ defined on $M^2$) of $\E$, and set 
 \[
    \lambda := \mu\left(
                    \varphi\left(\frac{\partial}{\partial x_1}\right),
                    \varphi\left(\frac{\partial}{\partial x_2}\right)
                  \right).
 \]
 Then $\lambda$ is an oriented 
 $\varphi$-function with respect to the
 orientations of $\E$ and $M^2$.
 Since 
 $\partial/\partial x_1$ is an extended 
 null vector field, 
 $\varphi(\partial/\partial x_1)$ vanishes on $\Sigma^1\cap U=\{\lambda=0\}$.
 Then
 by the well-known preparation theorem
 for $C^\infty$-functions, 
 there exists a section $\vect{e}_1$ of $\E$ 
 such that
 $\phi(\partial/\partial x_1)=\lambda \vect{e}_1$
 On the other hand, we set
 $\vect{e}_2:=\phi(\partial/\partial x_2)$.
 Then $\{\vect{e}_1,\vect{e}_2\}$ gives  
 a frame field of $\E$ on $U$ which is compatible 
 with the orientation
 of $\E$.
 In fact,
 \[
    \lambda =  \mu\left(
                 \varphi\left(\frac{\partial}{\partial x_1}\right),
                 \varphi\left(\frac{\partial}{\partial x_2}\right)
                \right)
            =\mu\bigl(\lambda\vect{e}_1,\vect{e}_2\bigr)
            =\lambda\mu(\vect{e}_1,\vect{e}_2),
 \]
 and hence $\mu(\vect{e}_1,\vect{e}_2)=1$.
 We set
 \[
     X=
        \xi_1 \frac{\partial }{\partial x_1}+
        \xi_2 \frac{\partial }{\partial x_2}
        \qquad\text{and}\qquad
     Y = \alpha_1 \vect{e}_1 + \alpha_2\vect{e}_2.
 \]
 Then it holds that
 \[
      \alpha_1 = \lambda \xi_1 ,\qquad \alpha_2 = \xi_2.
 \]
 Since 
 $\lambda$ vanishes on $\Sigma^{1}$
 and since
 $\partial/\partial x_2$ spans $T_p\Sigma^1$
 (cf.\ \ref{coord:2} of Proposition~\ref{thm:coord}),
 we have $\lambda(p)=\lambda_{x_2}(p)=0$,
 where $\lambda_{x_2}:=\partial\lambda/\partial x_2$.
 In particular, the equality
 $(\alpha_1)_{x_2}=\partial \alpha_1/\partial x_2=0$
 holds at $p$. 
 Since the equalities
 $(\alpha_1)_{x_1}=\lambda_{x_1} \xi_1=\dot\lambda\xi_1$
 also hold at $p$, we have that
 \begin{align*}
    \sgn\left(
    \det
      \pmt{
          (\alpha_1)_{x_1} & (\alpha_1)_{x_2}\\
          (\alpha_2)_{x_1} & (\alpha_2)_{x_2} \\
      }\right)&=
    \sgn\biggl(
    \dot \lambda \xi_1 (\alpha_2)_{x_2}
    \biggr)\\
      &=
    \sgn
    \biggl(
    \dot \lambda \xi_1 (\xi_2)_{x_2}
    \biggr)
      =
    \ind_p(X_1)\sgn\biggl(\dot \lambda \xi_1 \biggr).
 \end{align*}
 Here, we used the relation 
 $\ind_p(X_1)=\sgn(\xi_2)_{x_2}$.
 In fact, by \ref{coord:2} of Proposition~\ref{thm:coord},
 one can parametrize $\Sigma^1$ around $p$ as
 \[
      \Sigma^1\cap U=\{(x_1,x_2)=(f(t),t)\,;\,t\in I\},
 \]
 where $I$ is a sufficiently small interval including $0$
 and $f$ is a smooth function on $I$ such that $df(0)/dt=0$.
 That is, $t$ can be taken as a local coordinate system
 of $\Sigma^1$.
 Then there exists a smooth function 
 $\hat\xi:I\to \R$ such that
  \[
       X_1=\hat\xi\frac{d}{dt}
        =\hat\xi\left(\frac{df}{dt}\frac{\partial}{\partial
             x_1}+\frac{\partial}{\partial x_2}\right).
  \]
 The condition \ref{char:1} of Definition \ref{def:X} yields 
 that 
 \[
      \hat\xi = \xi_2,\qquad
      \frac{d\hat\xi}{dt}
       = \frac{df}{dt}\frac{\partial\xi_2}{\partial x_1}+
         \frac{\partial\xi_2}{\partial x_2}.
 \]
 Since $df(0)/dt=0$, we have
 \[
     \ind_p(X_1) = \sgn_{t=0}\left(\frac{d\hat\xi}{dt}\right)
                 = \sgn_p\left(
                           \frac{\partial\xi_2}{\partial x_2}
                          \right).
 \]

 Since $p\not \in \N_\delta(\Sigma^0)$
 for sufficiently small $\delta$,
 the characteristic vector field 
 $X$ points in the direction of $M^2_+=\{\lambda>0\}$ at $p$.
 So the equality
 \[
   \sgn(\xi_1)=
            \sgn(\dot \lambda)
 \]
 holds at $p$.
 Thus 
 $\dot \lambda(p) \xi_1(p)>0$ and
 \[
    \ind_p(Y)=\ind_p(X_1).
 \]
 Since $Z(X_1)\subset \AA_2$ 
 and
 $\chi(\Sigma^1)=0$,
 applying the Poincar\'e-Hopf index formula for 
 the vector field $X_1$ on
 $\Sigma^1$, we get the assertion.
\end{proof}
By \eqref{eq:dim2}, Proposition~\ref{prop:dim2_B} and
Proposition~\ref{prop:dim2_add},
formula \eqref{eq:main} 
follows immediately from the following assertion:
\begin{proposition}\label{prop:dim2_C}
 Let $p$ be an arbitrarily given $\A_3$-point.
 Then 
 \[
   \ind_p(Y)=
    \begin{cases}
      1 & (\mbox{if $p\in \AA_3^+$}), \\
     -1 & (\mbox{if $p\in \AA_3^-$}).
    \end{cases}
 \]
\end{proposition}
\begin{proof}
 We take a $\varphi$-adapted coordinate system
 $(U;x_1,x_2)$ centered at $p$
 which is compatible with the orientation of $M^2$.
 In particular, 
 $\tilde\eta:=\partial/\partial x_2$ is 
 an extended null  vector field on $U$, and
 $(\partial/ \partial x_2)_p\in T_p\Sigma^1$.
 Let $\mu$ be a local orientation of $\E$, and let
 $\lambda:=\mu\bigl(\varphi(\partial/\partial x_1),
                    \varphi(\partial/\partial x_2)\bigr)$.
 We set
 $\vect{e}_1:=\phi\left({\partial}/{\partial x_1}\right)$.
 Since $\phi(\tilde \eta)$ 
 vanishes on $\Sigma^1$,
 there exists a section $\vect{e}_2$ of $\E$ on $U$ such that
 $\phi(\partial/\partial x_2)=\phi(\tilde \eta)=\lambda \vect{e}_2$.
 Since 
 \[
    \lambda =  \mu\left(
                 \varphi\left(\frac{\partial}{\partial x_1}\right),
                 \varphi\left(\frac{\partial}{\partial x_2}\right)
                \right)
            =\mu\bigl(\vect{e}_1,\lambda\vect{e}_2\bigr)
            =\lambda\mu(\vect{e}_1,\vect{e}_2),
 \]
 we have  $\mu(\vect{e}_1,\vect{e}_2)=1$, which implies that
 $\{\vect{e}_1,\vect{e}_2\}$ forms a frame field of $\E$
 compatible with the orientation of $\E$.
 We set
 \[
    X= \xi_1 \frac{\partial }{\partial x_1}+
       \xi_2 \frac{\partial }{\partial x_2}.
 \]
 By \ref{char:2} and \ref{char:0} of Definition~\ref{def:X},
 $X_p\in T_p\Sigma^1$ and $X_p \ne 0$, and hence we have
 $\xi_1(p)=0$ and $\xi_2(p)\ne 0$.
 We now set
 \begin{equation}\label{eq:Ycomp1}
   Y=\alpha_1\vect{e}_1+
     \alpha_2\vect{e}_2.
 \end{equation}
 Then it holds that 
 $\alpha_1=\xi_1$ and $\alpha_2=\lambda \xi_2$.
 By \ref{char:2} of Definition~\ref{def:X},
 $X$ is tangent to $\Sigma^1$ near $p$.
 Since $\lambda$ vanishes along $\Sigma^1$,
 it holds that
 \[
    0=d\lambda(X)=\lambda_{x_1}\xi_1+\dot \lambda \xi_2
 \]
 on a sufficiently small neighborhood $p$ in $\Sigma^1$,
 where we used the fact that $\lambda_{x_2}=\dot \lambda$
 (cf.\ \ref{coord:3} of Proposition \ref{thm:coord}). 
 Since $d\lambda(X)$ vanishes along $\Sigma^1$
 and $\partial/\partial x_2\in T\Sigma^1$ at $p$,
 the fact
 $\xi_1(p)=\dot \lambda(p)=0$ yields that
 the equalities
 \begin{align*}
   0&=\frac{\partial d\lambda(X)}{\partial x_2} 
     =\lambda_{x_1x_2}\xi_1+\lambda_{x_1}(\xi_1)_{x_2}
         +\ddot \lambda \xi_2+\dot \lambda (\xi_2)_{x_2}\\
    &=\lambda_{x_1}(\xi_1)_{x_2}+\ddot \lambda \xi_2
 \end{align*}
 hold at $p$.
 Since $d\lambda(p)\ne 0$ and 
 $\lambda_{x_2}(p)=\dot \lambda(p)=0$,
 we can conclude that $\lambda_{x_1}(p)\ne 0$. 
 In particular, we have that
 \[
   (\xi_1)_{x_2}(p)=
     -\frac{\ddot \lambda(p) \xi_2(p)}{\lambda_{x_1}(p)}.
 \]
 Using the facts 
 $\lambda_{x_2}(p)=\dot \lambda(p)=0$,
 we have that
 \begin{align*}
  \ind_p(Y)
  &=\sgn
    \left(\det\pmt{(\xi_1)_{x_1}(p) & (\xi_1)_{x_2}(p) \\ 
              \lambda_{x_1}(p) \xi_2(p)
              & \xi_2(p)\lambda_{x_2}(p)} 
    \right)\\
  &=\sgn
        \left(
         \det\pmt{(\xi_1)_{x_1}(p) & (\xi_1)_{x_2}(p) \\ 
                \lambda_{x_1}(p) \xi_2(p) & 0} 
        \right)\\
  &=-\sgn\left(
            \xi_2(p)\lambda_{x_1}(p)\left(
             -\frac{\ddot \lambda(p) \xi_2(p)}{\lambda_{x_1}(p)}\right)
            \right)
   =\sgn\biggl(\xi_2(p)^2 \ddot \lambda(p)\biggr). 
 \end{align*}
 Since the sign of an $\A_3$-point
 coincides with the sign of $\ddot \lambda$, 
 Proposition \ref{prop:dim2_C} is proved.
\end{proof}

\section{The proof of the index formula}\label{sec:even4}

In this section, we prove 
our formula \eqref{eq:main} for $n$-manifolds ($n=2m\ge 4$).

Let $M^n$ be an oriented manifold, 
and $X$ a characteristic vector field associated to
a Morin homomorphism $\phi:TM^n\to \E$.
Suppose that $\E$ is oriented. 
Let $(U;x_1,\dots,x_n)$ be a $\phi$-adapted  coordinate system 
centered at an $\A_{2}$-point $p\in M^n$  
{\rm(}cf.\ Proposition \ref{thm:coord}{\rm)},
which is compatible with the orientation of 
$M^n$.
Suppose that $Y:=\varphi(X)$ vanishes at $p$.
Then $X$ has an expression
\begin{equation}\label{eq:x-component}
    X=\xi_1\frac{\partial}{\partial x_1}+\dots
              + \xi_n\frac{\partial}{\partial x_n}.
\end{equation}
By a property of $\phi$-adapted  coordinate systems,
\[
   \eta:=\frac{\partial}{\partial x_1}
\]
gives a null vector field.
By \ref{char:1} in Definition~\ref{def:X}, $\xi_1\ne 0$ holds.
Moreover, the fact $\phi(X_p)=0$
yields that
\begin{equation}\label{eq:xik}
  \xi_{1}(p)\ne 0,\quad \xi_j(p)=0\qquad (j=2,\dots,n).
\end{equation}

\begin{lemma}\label{thm:coord2}
 It holds that
 \[
   \ind_p(Y)=\sgn
           \biggl(\xi_{1}(p)\dot \lambda(p)\biggr)\ind_p(X_{n-1}).
 \]
\end{lemma}

\begin{proof}
 Let $\mu$ be an orientation of $\E$, 
 and set
 \[
    \lambda:=\mu\left(
              \varphi\left(\frac{\partial}{\partial x_1}\right),\dots,
              \varphi\left(\frac{\partial}{\partial x_n}\right)
                \right),
 \]
 which is an oriented $\varphi$-function on a neighborhood of $p$.
 We set
 \[
    \vect{e}_j:=\phi(\partial/\partial x_j)
            \qquad (j=2,\dots,n).
 \]
 Since $\tilde\eta_1=\partial/\partial x_1$ is an extended null 
 vector field,
 by the preparation theorem 
 for $C^\infty$-functions, we can write
 $\phi(\partial/\partial x_{1})=\lambda \vect{e}_{1}$,
 where $\vect{e}_1$ is a local section defined on a neighborhood of $p$.
 Since
 \[
   \lambda=
        \mu\left(
          \phi\left(\frac{\partial}{\partial x_1}\right),\dots,
          \phi\left(\frac{\partial}{\partial x_n}\right)\right)
          =\lambda\mu(\vect{e}_1,\dots,\vect{e}_n),
 \]
 we have 
 $\mu(\vect{e}_1,\dots,\vect{e}_n)=1$.
 In particular, $\{\vect{e}_1,\dots, \vect{e}_n\}$
 gives an oriented 
 frame on the vector bundle $\E$ around $p$.
 So we can write
 \begin{equation}\label{eq:y-repr}
    Y=\alpha_1\vect{e}_1+\cdots+\alpha_n\vect{e}_n,\qquad
     \mbox{where}\qquad
    \alpha_j = \begin{cases}
	         \xi_j\qquad &(j\neq 1),\\
	        \lambda\xi_1\qquad &(j=1).
	      \end{cases}
 \end{equation}
 We set
 \[
   J:=\det(\alpha_{ij})_{i,j=1,\dots,n},
    \qquad \alpha_{ij}:=\frac{\partial \alpha_i}{\partial x_j}.
 \]
 If $J(p)\ne 0$, it holds that 
 \begin{equation}\label{eq:Jacobi}
  \ind_p(Y)=\sgn\bigl(J(p)\bigr).
 \end{equation}
 By \eqref{eq:xik} and \eqref{eq:y-repr}, we have that
 \begin{equation}\label{eq:lv0}
  (\alpha_1)_{x_1}(p)=\xi_1(p)\lambda_{x_1}(p)\ne 0,\qquad
   (\alpha_1)_{x_2}(p)=\cdots=(\alpha_1)_{x_n}(p)=0.
 \end{equation}
 Then \eqref{eq:lv0} implies that
 \begin{align*}
    J(p) &= \det
     \begin{pmatrix}
      (\alpha_1)_{x_1} &  0 & \dots & 0 \\
      (\alpha_2)_{x_1} &  (\alpha_2)_{x_2} & \dots & (\alpha_2)_{x_n} \\
      \vdots & \vdots & \ddots & \vdots \\
      (\alpha_n)_{x_1} &  (\alpha_n)_{x_2} & \dots & (\alpha_n)_{x_n} 
     \end{pmatrix}\\
       &= (\alpha_1)_{x_1} 
     \det    
     \begin{pmatrix}
      (\alpha_2)_{x_2} & \dots & (\alpha_2)_{x_n} \\
     \vdots & \ddots & \vdots \\
       (\alpha_n)_{x_2} & \dots & (\alpha_n)_{x_n} 
     \end{pmatrix}\\
     &= (\alpha_1)_{x_1} 
     \det    
     \begin{pmatrix}
      (\xi_2)_{x_2} & \dots & (\xi_2)_{x_n} \\
     \vdots & \ddots & \vdots \\
       (\xi_n)_{x_2} & \dots & (\xi_n)_{x_n} 
     \end{pmatrix},
 \end{align*} 
 because of \eqref{eq:y-repr}.
 Thus, by \eqref{eq:xik},
\eqref{eq:Jacobi} and \eqref{eq:lv0},
we have that
 \begin{align*}
    \ind_p(Y) &= \sgn\bigl(J(p)\bigr)
              = \sgn\bigl((\alpha_1)_{x_1}\bigr)
                \sgn\frac{\partial(\xi_2,\dots,\xi_n)}{\partial (x_2,\dots,x_n)}\\
              &= \sgn\bigl(\xi_1(p)\lambda_{x_1}(p)\bigr)
                \ind_p(X_{n-1}).
 \end{align*}
\end{proof}

We now prove the formula \eqref{eq:main}.
Let $M^n$ ($n=2m$)
be a compact oriented  $n$-manifold without boundary,
and $\phi:TM^n\to \E$ be a Morin homomorphism,
where $\E$ is an oriented vector bundle.
We fix a characteristic vector field $X$ 
as in the previous section 
(cf.\ Proposition~\ref{prop:charcaristic}).
Take a section $Y$ of $\E$ as 
\[
   Y:=\phi(X).
\]
We denote by $Z(X)$ and $Z(Y)$ the set of zeros of
$X$ and $Y$, respectively.
The following assertion can be proved
as in Proposition \ref{prop:dim2_A}.

\begin{proposition}\label{prop:dim4_A}
 Let $Z(X_{n-j})$ $(j=0,1)$
 be the set of zeros 
 for $X_j$, where $X_n=X$.
 Then it holds that
 \begin{align}
  &Z(Y)\cap (M^n\setminus \Sigma^{n-1})=Z(X), \label{eq:4d1} \\
  &Z(Y)\cap (\Sigma^{n-1}\setminus \Sigma^{n-2})=Z(X_{n-1}). \label{eq:4d2}
 \end{align}
\end{proposition}

By \eqref{eq:Y}, it is sufficient to show the following
assertion. 

\begin{theorem}
 The following identity holds
 \[
   \sum_{p\in M^n}\ind_p(Y)
    =\chi(M^{2m}_+)-\chi(M^{2m}_-)+
     \sum_{j=1}^m \biggl( \chi(\AA_{2j+1}^+)-\chi(\AA_{2j+1}^-)\biggr).
 \]
\end{theorem}

We prove the theorem
by induction on the dimension $n=2m$.
We have already shown that the formula holds for $m=1$
in Section~\ref{sec:even2}.
So we now assume that the formula \ref{eq:main}
holds for $m-1$, and will prove the case for $m$.
Let 
\[
  \hat \phi:T\Sigma^{n-1}\to \hat \E
\]
be the reduction.
Then it induces again the second reduction
$\hat{\hat \phi}:T\Sigma^{n-2}\to \hat{\hat \E}.$
Since $\E$ is oriented, we can take an oriented 
$\phi$-function $\lambda:M^n\to \R$ satisfying
\eqref{eq:lambda2}.
By Proposition \ref{prop:k_even}, 
$\ddot \lambda$ is an oriented 
$\hat {\hat \phi}$-function
of $\hat{\hat \E}$ defined on $\Sigma^{n-2}$.
Since the restriction of $X$ to $\Sigma^{n-2}$
is a characteristic vector field of $\Sigma^{n-2}$,
the induction assumption yields that
\[
  \sum_{p\in \Sigma^{n-2}}\ind_p(Y)=
  \sum_{j=1}^m \biggl( \chi(\AA_{2j+1}^+)-\chi(\AA_{2j+1}^-)\biggr).
\]
On the other hand, as in Proposition \ref{prop:dim2_B},
one can prove the following assertion:

\begin{proposition}\label{prop:dim4_B}
 The first term of the right-hand side of \eqref{eq:Y} 
 in the introduction satisfies
 \begin{equation}\label{eq:1-st-4}
    \sum_{p\in M^n\setminus \Sigma^{n-1}}
     \ind_p(Y)=\chi(M^n_+)-\chi(M^n_-).
 \end{equation}
\end{proposition}

Now formula \eqref{eq:main} for the $2m$-dimensional case
reduces to the following assertion:

\begin{proposition}\label{prop:dim4_C}
 The second term 
 of the right-hand side of \eqref{eq:Y} 
 in the introduction satisfies
 \begin{equation}\label{eq:1-st-2}
  \sum_{p\in \Sigma^{n-1}\setminus \Sigma^{n-2}} \ind_p(Y)=0.
 \end{equation}
\end{proposition}

\begin{proof}
 We fix a point 
 $p\in \Sigma^{n-1}\setminus \Sigma^{n-2}$
 satisfying $Y_p=0$ arbitrarily.
 By property \ref{char:2} in Definition~\ref{def:X},
 there exists a vector field $X_{n-1}$ on $\Sigma^{n-1}$
 such that $Z(X_{n-1})=Z(Y)\cap (\Sigma^{n-1}\setminus \Sigma^{n-2})$.
 By Lemma \ref{thm:coord2},
 it holds that
 \[
    \ind_p(Y)=\ind_p(X_{n-1})
       \sgn\left(\dot \lambda(p)\xi_{1}(p)\right).
 \]
 By \ref{char:2} of Definition~\ref{def:X}, 
 $\xi_1 \partial/\partial x_1$ points into $M^n_+$ at $p$.
 Since $\partial/\partial x_1$ is an extended null vector
 field, $\dot \lambda\partial/\partial x_1$ points
 also into $M^n_+$ at $p$ (cf.\ Proposition 
 \ref{prop:k_odd}).
 Hence
 \[
        \sgn\bigl(\dot \lambda(p)\xi_1(p)\bigr)\ge 0,
 \]
 and
 $\ind_p(Y)=\ind_p(X_{n-1})$
 holds. 
 Since $Z(X_{n-1})=Z(Y)\cap (\Sigma^{n-1}\setminus \Sigma^{n-2})$
 and $\Sigma^{n-1}$ is odd dimensional,
 it holds that
 \[
   \sum_{p\in \Sigma^{n-1}\setminus \Sigma^{n-2}}
    \ind_p(Y)
       =
     \sum_{p\in \Sigma^{n-1}}
    \ind_p(X_{n-1})
       =
    \chi(\Sigma^{n-1})=0.
 \]
\end{proof}

\section{Applications}\label{sec:appl}

In this section, we shall give several applications of
the formula \eqref{eq:BW}:
recall that a $C^\infty$-map $f:M^{2m}\to N^{2m}$ 
between $2m$-manifolds is called a {\it Morin map\/} if
the corresponding bundle homomorphism $\varphi=df$
as in Example~\ref{ex:Morin}
admits only $\A_k$-singularities for  $k=2,\dots,2m+1$
(cf.\ Remark~\ref{rmk:Morin}).
\begin{theorem}[{\cite{N} and \cite{DF}}]\label{thm:quine}
 Let $M^{2m}$ and $N^{2m}$ be compact oriented $2m$-manifolds,
 and let $f:M^{2m}\to N^{2m}$ be a Morin map.
 Then it holds that
 \begin{equation}\label{eq:Quine}
    \deg(f)\chi(N^{2m})=
     \chi(M^{2m}_+)-\chi(M^{2m}_-)
     +\sum_{j=1}^m \chi(\AA_{2j+1}^+)-\chi(\AA_{2j+1}^-),
 \end{equation}
 where $\deg(f)$ is the topological degree
 of the map $f$, and
 $M^{2m}_+$ {\rm(}resp.\ $M^{2m}_-${\rm)}
 is the set of points at which
 the Jacobian of $f$  is 
 positive {\rm(}resp.\ negative{\rm)}.
\end{theorem}
This formula is a generalization of Quine's formula \cite{Q}
for Morin maps between $2$-manifolds (see also \cite{E}).
It should be remarked that
the numbering of Morin singularities
is different from the usual one (cf.\ Remark \ref{rmk:Morin}). 
For example, a fold (resp.\ a cusp) singularity is an $\A_2$-singular point 
(resp.\ an $\A_3$-singular point) in \eqref{eq:Quine}.
\begin{proof}[Proof of Theorem~\ref{thm:quine}]
 Let $\E$ be the pull-back of the tangent bundle $TN^{2m}$ of $N^{2m}$
 by $f$. 
 Then the map $f$ induces a bundle-homomorphism
 $\phi_f:=df:TM^{2m}\to \E$ %
 as in Example~\ref{ex:Morin}.
 Since $f$ is a Morin map, $\phi_f$ has only
 $\A_k$-points, and then the formula follows
 from \eqref{eq:main} using the fact that 
 $\chi^{}_{\E}=\deg(f)\chi(N^{2m})$.
\end{proof}

Next we give applications for immersed hypersurfaces in $\R^{2m+1}$. 
Let $M^{2m}$ be a compact oriented $2m$-manifold and
$f:M^{2m}\to \R^{2m+1}$  a wave front. 
Suppose that there exists a unit normal vector field
$\nu$ along $f$ defined on $M^{2m}$. 
Then it induces the Gauss map into the unit
$2m$-sphere
$\nu:M^{2m}\longrightarrow S^{2m}$,
and a family of wave fronts
\[
    f_t:=f+t \nu \qquad (t\in \R),
\]
each of which is called a 
{\it parallel hypersurface\/}
of $f$.
The Gauss map of $f_t$ is commonly equal to $\nu$
for all $t\in \R$.
The Gauss map $\nu$ can be considered as the limit 
$\lim_{t\to \infty} f_t/t$.

\begin{corollary}\label{thm:BW2}
 Let $M^{2m}$ be a compact oriented $2m$-manifold and
 $f:M^{2m}\longrightarrow \R^{2m+1}$
 an immersion.
 Suppose that the Gauss map
 $\nu$ is a Morin map.
 Then the singular set of $\nu$
 satisfies identity  \eqref{eq:BW} in the introduction, 
 where $M^{2m}_-$ is the set of points at which
 the Gauss-Kronecker curvature of $f$ 
 {\rm(}i.e.\ the determinant of the shape operator{\rm)}
 is negative.
\end{corollary}

This formula is a generalization 
of the Bleeker-Wilson formula for Gauss maps 
of immersed surfaces in $\R^{3}$.

\begin{proof}[Proof of Corollary~\ref{thm:BW2}]
 We apply formula
 \eqref{eq:Quine} for the Gauss map $\nu$ of
 the immersion $f$.
 Then we have that
 \[
   2(\deg \nu) =\chi(M^{2m}_+)-\chi(M^{2m}_-)
     +
    \sum_{j=1}^m \biggl( \chi(\AA_{2j+1}^+)-
         \chi(\AA_{2j+1}^-)\biggr).
 \]
 Since $f$ is an immersion, it is well-known that
 $2(\deg \nu)$ is equal to $\chi(M^{2m})$.
 
 Next, we show that $M^{2m}_+$ (resp.\ $M^{2m}_-$)
 coincides with the set where the Gauss-Kronecker 
 curvature is positive (resp.\ negative):
 Let $ds^2$ be the induced Riemannian metric
 on $M^{2m}$ by the immersion $f$,
 and let $\vect e_1,\cdots,\vect e_{2m}$ be 
 an oriented local orthonormal
 frame field on $M^{2m}$ such that
 \[
    d\nu(\vect e_j)=-\mu_j df(\vect e_j)\qquad (j=1,\dots,2m),
 \]
 that is, $\vect e_1$,\dots, $\vect e_{2m}$ are
 eigenvector fields of the shape operator of $f$,
 and $\mu_1$, \dots, $\mu_{2m}$ are principal curvatures.
 Then we have that
 \begin{equation}\label{eq:GK}
  \lambda:=\det\bigl(
         d\nu(\vect e_1),\dots,
	 d\nu(\vect e_{2m}),\nu\bigr) 
          =\prod_{j=1}^{2m} \mu_j  =K,
 \end{equation}
 where $K:=\mu_1\cdots \mu_{2m}$
 is the Gauss-Kronecker curvature of $f$.
 This $\lambda$ is positive (resp.\ negative) 
 if and only if  $K>0$ (resp.\ $K<0$), which proves the assertion.
\end{proof}

Next, we show the following.

\begin{theorem}
\label{thm:A}
 Let $M^{2m}$ be a compact oriented $2m$-manifold  and
 $f:M^{2m}\to \R^{2m+1}$  a wave front.
 Suppose that $f$  admits only $\A_{k}$-front singularities
 {\rm(}$2\leq k \leq 2m+1${\rm)},
 as defined in
 Definition \ref{def:front}.
 Then the singular set of $f$ satisfies the identity
 \begin{equation}\label{eq:main2}
   2\deg(\nu) 
    =\chi(M^{2m}_+)-\chi(M^{2m}_-)+
      \sum_{j=1}^m \biggl( \chi(\AA_{2j+1}^+)-\chi(\AA_{2j+1}^-)\biggr),
 \end{equation}
 where $\deg(\nu)$ is the degree of the Gauss map 
 $\nu\colon{}M^{2m}\to S^{2m}$
 induced by $f$, and $\chi(M^{2m}_+)$ 
 {\rm(}resp.\ $\chi(M^{2m}_-)${\rm)} is the Euler characteristic
 of the subset $M^{2m}_+$ {\rm(}resp.\ $M^{2m}_-${\rm)} 
 of $M^{2m}$ at which
 \[
     \lambda:=\det(f_{x_1},\cdots,f_{x_{2m}},\nu)
 \]
 is positive {\rm(}resp.\ negative{\rm)}
 for an oriented local coordinate system $(x_1,\dots,x_{2m})$,
 where $f_{x_j}=\partial f/\partial x_j$.
\end{theorem}

This formula is independent
of the index formula for the Gauss map $\nu$ 
(cf.\ Theorem~\ref{thm:quine}).
In fact, the singular set of $f$ does  not
coincide with that of its Gauss map in general.

\begin{proof}[Proof of Theorem \ref{thm:A}]
 We apply \eqref{eq:main} 
 for the bundle homomorphism 
 \[
   \varphi_f:=df:TM^{2m}\longrightarrow \E_{f}
 \]
 as in Example~\ref{ex:wave}. 
 Then it is sufficient to show that $\chi_{\E_f}$ is 
 equal to $2\deg(\nu)$.
 Let $\xi$ be a vector field on the unit $2m$-sphere
 $S^{2m}$.
 By parallel transport, $\xi_q$ ($q\in S^{2m}$)
 can be considered as a vector in $\E_p$ for $p\in \nu^{-1}(q)$.
 Thus, $\xi$ induces a 
 section $\tilde \xi$ of $\E$ defined on $M^{2m}$. 
 Then the equalities
 \[
   \chi^{}_{\E_f}=\sum_{p\in M^{2m}}\ind_p(\tilde \xi)
      =\deg(\nu)\sum_{q\in S^{2m}}\ind_q(\xi)
      =\deg(\nu)\chi(S^{2m})=2\deg(\nu)
 \]
 hold, which proves the identity.
\end{proof}

Next, we give an application to
parallel hypersurfaces of strictly convex 
hypersurfaces. 
\begin{theorem}\label{thm:BW}
 Let $S^{2m}$ be the unit $2m$-sphere,
 and let
$
    f:S^{2m}\longrightarrow \R^{2m+1}
$
 be a strictly convex immersion, that is,
 the Gauss map $\nu\colon{}S^{2m}\to S^{2m}$ is a diffeomorphism.
 Let $t\in \R$ be a value such that
 the parallel hypersurface
 \[
   f_t:S^{2m}\longrightarrow \R^{2m+1}
 \]
 has only $\A_k$-singularities
 $(k=2,\cdots,2m+1)$.
 Then the singular set of $f_t$
 satisfies \eqref{eq:BW}
 and 
 $1/K_t$ can be extended as a $C^\infty$-function on
$S^{2m}$ and gives 
an  oriented $\phi_t$-function for $\phi_t=df_t$
 {\rm(}cf.\ Definition~\ref{def:phi-function}{\rm)}, where
 $K_t$ is the Gauss-Kronecker curvature of $f_t$.
\end{theorem}
The corresponding assertion for a convex surface $f:S^{2}\to \R^{3}$
is given by Martinez-Maure \cite{MM}
under the generic assumption that
the Gaussian curvature is unbounded at the singular set of $f_t$,
and proved in \cite{SUY6} for the general case.
The above formula is a generalization of it.
\begin{proof}[Proof of Theorem~\ref{thm:BW}]
 We apply Theorem \ref{thm:A} for the
 bundle homomorphism 
 $\phi_t=df_t:TS^{2m}\to \E_{f_t}$.
 Since $f$ is convex, 
 the Gauss map $\nu\colon{}S^{2m}\to S^{2m}$
 is of degree one.
 Since $f=f_0$ is an immersion, and the Gauss map
 $\nu$ is common in the parallel 
family $\{f_t\}_{t\in \R}$,
 we have that
 \begin{align*}
   \chi(M_+^{2m})+\chi(M_-^{2m})&=\chi(S^{2m})
    =2\deg(\nu)\\
     &=\chi(M_+^{2m})-\chi(M_-^{2m})
      +
      \sum_{j=1}^m 
       \biggl(
         \chi(\AA_{2j+1}^+)-\chi(\AA_{2j+1}^-)
       \biggr),
 \end{align*}
 where $M^{2m}_{+}=S^{2m}_+$ (resp.\ $M^{2m}_{-}=S^{2m}_-$)
 is the set where $\lambda_t>0$
(resp.\ $\lambda_t<0$).
Here, $\lambda=\lambda_t$ 
 is the function as in the statement of 
 Theorem \ref{thm:A}.
 Moreover, since $\nu$ is an immersion,
 one can take the Riemannian metric $d\sigma^2$
 on $S^{2m}$ as the pull-back of the canonical 
 metric of $S^{2m}$ by $\nu$,
 and let $\{\vect{e}_1,\dots,\vect{e}_{2m}\}$ be 
 an oriented local orthonormal
 frame field on $S^{2m}$ with respect to $d\sigma^2$ such that
 \[
    df(\vect e_j)=-(1/\mu_j) d\nu(\vect e_j)\qquad (j=1,\dots,2m),
 \]
 that is, $\vect e_1$,\dots, $\vect e_{2m}$ are
 eigenvector fields of the shape operator of $f$.
 Since
 \[
    df_t(\vect{e}_j)= df(\vect{e}_j) + td\nu(\vect{e}_j)
                    = -\left(\frac{1}{\mu_j}-t\right)d\nu(\vect{e}_j),
 \]
 the Gauss-Kronecker curvature $K_t$ of $f_t$ is expressed as
 \[
    K_t  =\left(
              \prod_{j=1}^{2m}\left(\frac{1}{\mu_j}-t\right)
          \right)^{-1}.
 \]
 On the other hand,
 \begin{align*}
    \lambda_t&:=
    \det \bigl(df_t(\vect{e}_1),\dots,
               df_t(\vect{e}_{2m}),\nu
            \bigr)\\
      &= \left(
              \prod_{j=1}^{2m}\left(\frac{1}{\mu_j}-t\right)
          \right)
           \det \bigl(d\nu(\vect{e}_1),\dots,
               d\nu(\vect{e}_{2m}),\nu
            \bigr)\\
&=\frac{1}{K_t}
           \det \bigl(d\nu(\vect{e}_1),\dots,
               d\nu(\vect{e}_{2m}),\nu
            \bigr)=K\left(\frac{1}{K_t}\right),
 \end{align*}
 which implies that $1/K_t$ is an 
oriented $\varphi_t$-function,
since
$
K$ is positive because of the convexity of $f$,
 where $\varphi_t=df_t$.
\end{proof}

Now we consider the singularities of vector fields on $M^{2m}$.
Let $D$ be an arbitrary linear connection on $M^{2m}$
and $X$ a vector field defined on $M^{2m}$.
One can apply \eqref{eq:main}
for the bundle homomorphism
\[
   \phi_X:TM^{2m}\ni v \longmapsto D_vX\in TM^{2m}
\]
if $\phi_X$ admits only $\A_k$-singularities
and get \eqref{eq:BW}, 
where
$M^{2m}_+$ is the set of points where
\[
   (D_{v_1}X,\dots,D_{v_{2m}}X)
\]
forms a positive frame
for a given locally defined
positive frame $v_1,\dots,v_{2m}$ on $T_pM^{2m}$.
In \cite{SUY5}, this map was introduced
on a Riemannian $2$-manifold,
and  we called the singular points 
of $\phi_X$ the {\em irrotational points\/} there.
However, it would be better to call them
the {\it $\A_k$-singular points of the vector field\/} 
with respect to the connection $D$.
In fact, the singular set of 
$\phi_X$ has no relation 
with the rotations
of the vector fields. 

At the end of this section, we give an application 
for the Blaschke normal maps for strictly convex hypersurfaces:
we fix a strictly convex immersion
\[
   f:S^{2m}\longrightarrow \R^{2m+1}.
\]
Then there exists a unique vector field $\xi$ along $f$
satisfying the following two properties, which is called
the {\it affine normal vector field}:
\begin{enumerate}
 \item the linear map
       \[
         S:TS^{2m}\ni v \longmapsto D_v\xi
       \]
       gives an endomorphism on $TS^{2m}$, that is,
       $S(v):=D_v\xi$ is tangent to 
       $f(S^{2m})$ for each $v$, where
       $D$ is the canonical affine connection on $\R^{2m+1}$,
 \item there exists a unique covariant symmetric tensor $h$ 
       such that 
       \[
           D_X df(Y)-h(X,Y)\xi
       \]
       gives a tangential vector field on $f(S^{2m})$ for 
       any vector fields
       $X$ and $Y$ on $S^{2m}$. 
       Since $f$ is strictly convex, 
       $h$ is positive definite.
       Then the $2m$-form $\Omega$ defined by
       \[
        \Omega(X_1,\dots,X_{2m}):=
          \det\bigl(df(X_1),\dots,df(X_{2m}),\xi\bigr)
       \]
       coincides with the volume element associated to $h$,
       where $X_1$,\dots, $X_{2m}$ are vector fields on 
       $S^{2m}$ and 
       ``$\det$'' denotes the canonical volume form of  
       $\R^{2m+1}$.
\end{enumerate}
The vector field $\xi$ induces a map
\begin{equation}\label{eq:blaschke}
    \xi:S^{2m}\ni p\longmapsto \xi_p\in \R^{2m+1},
\end{equation}
which is called the {\it Blaschke normal map\/} of $f$.
The following assertion holds as in the case of $m=1$
(cf.\ \cite[Lemma 3.1]{SUY5}).

\begin{lemma}
 The Blaschke normal map $\xi$ gives a wave front.
\end{lemma}
\begin{proof}
 Consider a non-zero section
 \[
    L:S^{2m}\ni p \longmapsto (\xi_p,\nu_p)\in T^*\R^{2m+1}
           =\R^{2m+1}\times (\R^{2m+1})^*,
 \]
 where $(\R^{2m+1})^*$ is the dual vector space of $\R^{2m+1}$,
 and $\nu:S^{2m}\to (\R^{2m+1})^*$ is the map
 defined by
 \[
    \nu_p(\xi_p)=1,\qquad \nu_p(df(T_pS^{2m}))=\{0\}
         \qquad (p\in S^{2m}),
 \]
 which is called the {\it conormal map\/} of $f$.
 By definition, $L$ induces an isotropic map
 of $S^{2m}$ into the projective cotangent
 bundle $P(T^*\R^{2m+1})=\R^{2m+1}\times P^*(\R^{2m+1})$
 with the canonical contact structure.
 Take a local coordinate system $(x_1,\dots,x_{2m})$
 of $S^{2m}$.
 Then we have that
 \begin{align*}
  \nu_{x_i}(f_{x_j})
    &=(D_{\partial/\partial x_i}\nu)(f_{x_j})
     =\frac{\partial}{\partial x_i}\nu(f_{x_j})
         -\nu(D_{\partial/\partial x_i}f_{x_j}) \\
    &=-\nu(D_{\partial/\partial x_i}f_{x_j})\\
    &=
      -\nu\left(D_{\partial/\partial x_i}f_{x_j}
      -h\left(\frac{\partial}{\partial x_i},
      \frac{\partial}{\partial x_j}\right)\right)+
      h\left(\frac{\partial}{\partial x_i},
      \frac{\partial}{\partial x_j}\right)\\
    &=
      h\left(\frac{\partial}{\partial x_i},
      \frac{\partial}{\partial x_j}\right)
      \qquad (i,j=1,\dots,2m).
 \end{align*}
 Since $h$ is positive definite, one can show that
 $\nu_{x_1},\dots,\nu_{x_{2m}}$ are linearly independent.
 Moreover, since $\nu(T_pS^{2m})=\{0\}$ for each $p\in S^{2m}$,
 $\nu,\nu_{x_1},\dots,\nu_{x_{2m}}$ are linearly independent.
 In particular, the map $L$ induces a Legendrian immersion,
 which proves the assertion.
\end{proof}

The following assertion is
a hypersurface version of \cite[Theorem 3.2]{SUY5}.

\begin{theorem}
\label{thm:B}
 Let $S^{2m}$ be the $2am$-sphere
 and $f:S^{2m}\to \R^{2m+1}$ a strictly convex immersion.
 Suppose that the Blaschke normal map 
 $\xi:S^{2m}\to\R^{2m+1}$ 
 {\rm(}cf.\ \eqref{eq:blaschke}{\rm)}
 admits only $\A_k$-front singularities
 for $2\leq k\leq 2m+1$.
 Then the singular set of $\xi$ satisfies 
 \eqref{eq:BW}, where $M^{2m}_+(=S^{2m}_+)$
 {\rm(}resp.\ $M^{2m}_-(=S^{2m}_-)${\rm)}
 is the subset of $S^{2m}$ at which the determinant
 of the affine shape operator 
 {\rm(}called the affine Gauss-Kronecker curvature{\rm)}
 is positive {\rm(}resp.\ negative{\rm)}, and $\AA^+_{2j+1}$
 {\rm(}resp.\ $\AA^-_{2j+1}${\rm)}
 is the set of positive {\rm(}resp.\ negative{\rm)}
 $\A_{2j+1}$-front singular points of $\xi$
 for each $j=1,\dots,m$. 
\end{theorem}

\begin{proof}
 If the singular points of $\xi$ consist only of $\A_k$-points
 ($2\leq k\leq 2m+1$),
 the affine shape operator
 \[
    S\colon{}TS^{2m}\ni v\longmapsto D_v\xi\in f^*T\R^{2m+1}
 \]
 gives a Morin homomorphism.
 Applying \eqref{eq:main} for $S$, we get Theorem \ref{thm:B}.
\end{proof}

Finally, we give an example which illustrates Theorem \ref{thm:B}:
Consider a plane curve
\[
   \gamma(t)=(1-2\epsilon \sin t)\pmt{\sin t \\ \cos t}
            \qquad \left(-\frac{\pi}2\le t\le \frac{\pi}2\right),
\]
which lies on the upper-half plane and gives a
convex curve if $0\le \epsilon<1/4$.
Rotating it around the horizontal axis, we get a
rotationally symmetric strictly convex surface in $\R^3$.
The left hand side of Figure \ref{fig:normal} indicates
the curve $\gamma$ for $\epsilon=17/80$, and the
right hand side of Figure \ref{fig:normal} gives
the profile curve of the Blaschke normal map $\xi$ 
of the surface for $\epsilon=17/80$.
As shown in Figure \ref{fig:normal} (right), 
$\xi$ has no swallowtails (i.e. it has no
$A_3$-points), and our formula implies that
the Euler number $\chi(M_-^{2})$ vanishes. 
In fact, the set $\xi(M_-^{2})$ gives a cylindrical strip
if one rotates the profile curve of $\xi$ around the
horizontal axis.

\begin{figure}
 \begin{center}
  \includegraphics[width=4.5cm]{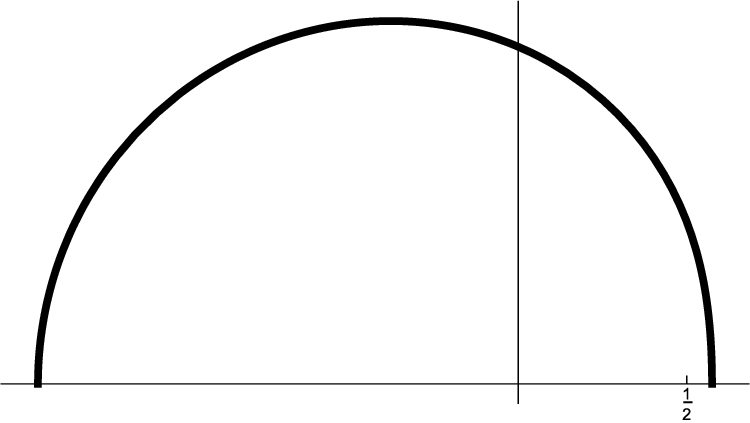}
  \hspace{2em}
  \includegraphics[width=4.5cm]{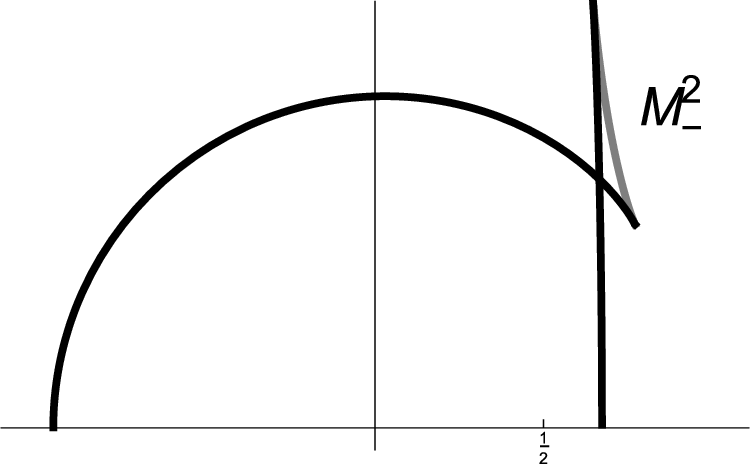}
 \end{center}
 \caption{%
   The curve $\gamma$ (left) and 
   the profile curve of $\xi$ (right).
 }\label{fig:normal}
\end{figure}

\section{Coherent tangent bundles induced by 
Kossowski metrics}\label{sec:Kossowski}

In this section, we introduce a class of positive semi-definite metrics 
called Kossowski metrics
describing the properties of wave fronts 
 intrinsically.
This class of metrics was defined by Kossowski \cite{K}
for $2$-dimensional manifolds.
In \cite{HHNSUY}, it was shown that
each Kossowski metric induces
a coherent tangent bundle,
and the formulas \eqref{eq:B} and 
\eqref{eq:A} for the metric were proved.
Our purpose is
to 
generalize the results of \cite{HHNSUY}
to higher dimensional cases,
that is,
we will give an application of
the formula \eqref{eq:main} for Kossowski metrics.

We now fix  an $n$-manifold $M^n$,
and a positive semi-definite metric $ds^2$
on $M^n$.  A point $p\in M^n$ is called a 
{\it singular point\/} of 
the metric $ds^2$ if the metric is 
not positive definite at $p$.
We denote by $\X$ the set of smooth vector fields
on $M^n$, and by $C^{\infty}(M^n)$ the set of $\R$-valued smooth functions on $M^n$.

We set
$\inner{X}{Y}:=ds^2(X,Y)$.
Kossowski \cite{K3} defined a map
$\Gamma:\X\times \X\times \X\to C^{\infty}(M^n)$
as 
\begin{multline}\label{eq:Gamma}
\Gamma(X,Y,Z):=\frac12
 \biggl(
 X\inner{Y}{Z}+Y\inner{X}{Z}-Z\inner{X}{Y}\\
 +\inner{[X,Y]}{Z}-\inner{[X,Z]}{Y}
 -\inner{[Y,Z]}{X}
 \biggr).
\end{multline}
We call $\Gamma$ the
{\it Kossowski pseudo-connection}.
(Kossowski \cite{K3} called $\Gamma$ the dual connection
of the Levi-Civita connection on 
$M^n\setminus \Sigma^{n-1}$, where $\Sigma^{n-1}$ is the singular set of $ds^2$.)
It was introduced by Kossowski (cf.\ \cite{K3}, \cite{K2} 
and \cite{K}), and  plays an important role to show 
a realization theorem of generic singularities of
Kossowski metrics as first fundamental forms
of wave fronts in $\R^3$.  
If the metric $ds^2$ is positive definite, then  
the equality
\begin{equation}\label{eq:GN}
\Gamma(X,Y,Z)=\langle \nabla_XY,Z\rangle
\end{equation}
holds, where
$\nabla$ is the Levi-Civita connection of $ds^2$.
One can easily check the following two identities (cf.\ \cite{K})
\begin{align}\label{eq:1}
&X\langle Y,Z\rangle=\Gamma(X,Y,Z)+\Gamma(X,Z,Y),\\
\label{eq:2}
&\Gamma(X,Y,Z)-\Gamma(Y,X,Z)=\langle [X,Y],Z\rangle.
\end{align}
The equation \eqref{eq:1} corresponds to the condition 
that  $\nabla$ is a metric connection, and
the equation \eqref{eq:2} corresponds to 
the condition that $\nabla$ is torsion free.
The following assertion can be also
easily verified: 

\begin{prop}[Kossowski \cite{K3}, \cite{K}]\label{prop:K-conn}
For each $Y\in \X$  and for each $p\in M^n$,
the map
\[
     T_pM^n\times T_pM^n\ni (v_1,v_2)\longmapsto
           \Gamma(V_1,Y,V_2)(p)\in \R
\]
is a well-defined bi-linear map, 
where $V_j$ $(j=1,2)$
are vector fields of $M^n$ 
satisfying $v_j=V_j(p)$.
\end{prop}

For each $p\in M^n$, the subspace
\begin{equation}\label{eq:N}
 N_p:=\biggl\{v\in T_pM^n\,;\, ds^2(v,w)=0
\mbox{ for all $w\in T_pM^n$}
\biggr\}
\end{equation}
is called the {\it null space} at $p$.
A non-zero vector which belongs to $N_p$ is called
a {\it null vector\/} at $p$.

\begin{lemma}[Kossowski \cite{K3}, see also \cite{HHNSUY}]
\label{lem:K-conn}
Let $p$ be a singular point of $ds^2$.
Then the Kossowski pseudo-connection $\Gamma$
induces a tri-linear map
 \[
    \hat \Gamma_p:T_pM^n\times T_pM^n\times N_p
            \ni (v_1,v_2,v_3) \longmapsto
          \Gamma(V_1,V_2,V_3)(p)\in \R,
 \]
where
$V_j$ $(j=1,2,3)$ are vector fields of $M^n$
such that $v_j=V_j(p)$.
\end{lemma}

\begin{definition}\label{def:adms}
 A singular point $p$ of the metric $ds^2$
 is called 
 {\it admissible\footnote{
 The notion of
 admissibility was introduced by 
 Kossowski \cite{K}. He called it
 $d(\langle,\rangle)$-{\it flatness}.
}%
\/} 
if $\hat \Gamma_p$ in Lemma~\ref{lem:K-conn} vanishes.
If each singular point of $ds^2$
is admissible, then $ds^2$ is called
an {\it admissible metric}.
\end{definition}

\begin{definition}\label{def:frontal}
An  admissible metric $ds^2$ defined on $M^n$
is called a {\it frontal metric\footnote{
\rm It is called a discriminant transverse 
metric in \cite{K}. }\/} if for each $p\in M^n$
there exists a
local coordinate
 system $(U;x_1,\dots,x_n)$
and a $C^\infty$-function
 $\lambda$ on $U$ such that
 \begin{equation}\label{eq:lambda}
  \det(g_{ij})=\lambda^2,
 \end{equation}
 where
 $ds^2=\sum_{i,j=1}^ng_{ij}dx_idx_j$
 is a local expression of the metric $ds^2$ on $U$
and $\det(g_{ij})$ is the determinant of the $n\times n$
matrix $(g_{ij})_{i,j=1,\dots,n}$.
\end{definition}

We remark that the condition \eqref{eq:lambda}
is independent of the choice of 
local coordinate systems.
If $f:M^n\to \R^{n+1}$ is a front,
then the induced metric $ds^2(:=df\cdot df)$ 
on $M^n$  is a frontal  metric 
(cf.\ \cite[Prop. 2.11]{HHNSUY}).

\begin{definition}\label{def:a2a3}
A singular point $p$ of a given frontal metric
is called {\it non-degenerate\/} or {\it generic}
(cf.\ \cite{K}) if its exterior 
derivative $d\lambda$ does not vanish at $p$, 
where $\lambda$ is the function as in
\eqref{eq:lambda}. A frontal metric $ds^2$
is called a {\it Kossowski metric\/} if
all of the singular points of the metric are
non-degenerate.
\end{definition}

One can easily check that each singular 
point of a Kossowski metric is non-degenerate,
and the singular set 
(denoted by $\Sigma^{n-1}$)
consists of a hypersurface
of $M^n$.
Moreover, the function $\lambda$ 
changes sign across $\Sigma^{n-1}$.
In particular, 
a $C^\infty$-function $\lambda$ satisfying 
\eqref{eq:lambda} 
is uniquely determined up to
the sign.

\begin{definition}[cf.\ \cite{HHNSUY}]\label{def:adms-coord}
 Let $ds^2$ be a Kossowski metric on $M^n$.
 A local coordinate system  $(U; x_1,\dots,x_n)$ of $M^n$
 is called {\it adjusted\/} at a singular point $p\in U$ 
 if  
 \[
     \partial_n:=\partial/ \partial x_n
 \]
 belongs to $N_p$. Moreover, if 
 $(U; x_1,\dots,x_n)$ is 
 adjusted at each singular point  of $U$, 
 it is called an 
 {\it adapted local coordinate system\/} of $M^n$.
\end{definition}

Since the singular set $\Sigma^{n-1}$ 
of a Kossowski metric is a hypersurface 
in $M^n$, one can easily prove
the existence of an adapted local coordinate system
at each singular point.
We are interested in the class of Kossowski metrics
because of the following fact:

\begin{proposition}\label{prop:frontal22}
Let
$
(\E,\phi,\langle,\rangle,D)
$
be a coherent tangent bundle 
{\rm(}see the introduction{\rm)} on a manifold $M^n$.
Then the induced metric $ds^2:=\phi^*\langle,\rangle$
is a frontal metric.
Moreover, if $\phi$ admits only non-degenerate singular 
points, then $ds^2$ is a Kossowski metric on $M^n$.
\end{proposition}

\begin{proof}
 The admissibility of the metric follows from the 
 identity
 \[
  \Gamma(X,Y,Z)=\langle D_X\phi(Y),\phi(Z)\rangle
   \qquad
 (X,Y,Z\in \X).
 \]
 On the other hand, for each  $p\in M^n$,
 one can take an orthonormal frame field
 $(\vect{e}_1,\dots,\vect{e}_n)$ of $\E$ on
 a coordinate neighborhood $(U;x_1,\dots,x_n)$ of $p$.
 Let $\theta_1,\dots,\theta_n$ be 
 the dual frame field of $(\vect{e}_1,\dots,\vect{e}_n)$. 
 Then
 $\mu:=\theta_1\wedge\cdots \wedge\theta_n$
 gives an orientation of $\E$
 on $U$,
 and there exists a smooth 
 function $\lambda\in C^\infty(U)$ such that
 \[
   \phi^*\mu=\lambda dx_1\wedge\cdots \wedge dx_n.
 \]
 If we write $ds^2=\sum_{i=1}^n g_{ij}dx_idx_j$ on $U$,
 then we have that
 \begin{equation}\label{eq:lambda3}
  |\lambda| =
     \sqrt{\det(g_{ij})},
 \end{equation}
 since $\phi^*\mu$ gives a Riemannian volume element
 on $U\setminus \Sigma^{n-1}$.
 Thus $\lambda^2$ coincides with $\det(g_{ij})$, 
 which implies that $ds^2$ is a frontal metric.
 Then the final assertion follows immediately 
 by
 comparing the definitions of non-degeneracy
 of singular points for $\phi$ and for $ds^2$. 
\end{proof}

\begin{example}\label{ex:conf}
 A Riemannian $n$-manifold $(M^n,g)$ 
 ($n\ge 3$) is called {\em conformally flat\/}
 if for each point $p\in M^n$, 
 there exists a neighborhood $U(\subset M^n)$ 
 of $p$ and a $C^\infty$-function $\sigma$
 on $U$ such that $e^{2\sigma}g$ is a metric with vanishing
 sectional curvature.
 When $n\ge 4$, $(M^n,g)$ is conformally flat
 if and only if the conformal curvature tensor 
 \begin{equation}\label{eq:weyl}
  W_{ijkl}:=R_{ijkl}+
   (B_{ik}g_{jl}-B_{il}g_{jk}+B_{jl}g_{ik}-B_{jk}g_{il})
      +\frac{S_g}{n(n-1)}(g_{ik}g_{jl}-g_{il}g_{jk})
 \end{equation}
 vanishes identically on $M^n$, where
 $(x_1,\dots,x_n)$ is a local coordinate system of $M^n$, 
 \begin{equation}\label{eq:schouten}
  B:=\sum_{i,j=1}^nB_{ij}
   dx_i\otimes dx_j,
   \quad
   B_{ij}:=\frac{1}{n-2}\left(R_{ij}-\frac{S_g g_{ij}}{2(n-1)}
       \right)
 \end{equation}
 is called the {\em Schouten tensor}, 
 $g_{ij}$, $R_{ijkl}$, $R_{ij}$ are the components of 
 the metric $g$, the curvature tensor of $g$,
 and the Ricci tensor of $g$ respectively, 
 and $S_g$ denotes the scalar curvature.
 When $n=3$, $(M^3,g)$ is conformally flat
 if and only if 
 $B$ in \eqref{eq:schouten}  is a Codazzi tensor, that is, 
 $\nabla B$ is a symmetric $3$-tensor, where $\nabla$ is the Levi-Civita
 connection of $(M^3,g)$.
 (When $n\ge 4$, conformal flatness implies that
 $B$ is a Codazzi tensor because of the second Bianchi identity.)
 We denote by $(g^{ij})_{i,j=1}^n$ the
 inverse matrix of $(g_{ij})_{i,j=1}^n$, and
 set
 \begin{equation}\label{hatA}
  \check B:=\sum_{i,j,a} g^{ia}B_{aj}\,
   \frac{\partial}{\partial x_i}\otimes dx_j
 \end{equation}
 which gives a $(1,1)$-tensor of $M^n$, and
 it induces a bundle homomorphism
 \begin{equation}\label{eq:schouten2}
   \check B:T_pM^{n}\ni v \mapsto 
   \check B_p(v)\in T_pM^{n}\qquad (p\in M^n).
 \end{equation}
 Since $B$ in \eqref{eq:schouten}  is a Codazzi tensor,
 $\check B$ satisfies the torsion free
 condition \eqref{eq:torsion} 
 with respect to
 $\nabla$
 (cf.\ \cite{LUY}),
 In particular, $\check B:TM^{n}\to 
 (TM^{n},g,\nabla)$ gives 
 a structure of a
 coherent tangent bundle.
 The pull-back of the Riemannian metric $g$ by $\check B$
 is given by
 \begin{equation}\label{eq:tilde}
  \check g:=\sum_{i,j,a,b} B_{ia}B_{jb}g^{ab}\,
     dx_idx_j.
 \end{equation}
 It is a remarkable fact that $\check g$ gives a 
 new conformally flat metric on 
 $M^n\setminus \Sigma^{n-1}$ (cf.\ \cite{LUY}).
 This new metric $\check g$ is called the {\it dual 
 metric} of $g$. 
 By Proposition \ref{prop:frontal22},
 $\check g$ gives an example of a 
 frontal metric.
 The points where $\check g$ is not positive definite
 correspond exactly to the singular points 
 of
 the bundle homomorphism $\check B$.
 We call  $A_k$-points of the bundle homomorphism
 $\check B$ the {\it $A_k$-points} of the dual metric.
\end{example}

As a converse of Proposition \ref{prop:frontal22}, 
the following assertion holds.

\begin{theorem}\label{thm:kossowski}
Let $ds^2$ be a Kossowski metric on an $n$-manifold $M^n$.
Then there exists a coherent tangent bundle
$
\phi:TM^n\to (\E,\inner{}{},D)
$
such that $\phi^*\inner{}{}$ coincides with $ds^2$.
\end{theorem}

The case of $n=2$ has already been 
proved in \cite{HHNSUY}, and 
this theorem is a generalization of it.
We fix an adapted local coordinate system 
$(U;x_1,\dots,x_n)$ arbitrarily.
We now carry out the
Schmidt orthogonalization for the frame
\[
  \partial_1
       :=\frac{\partial}{\partial x_1}\,\,,\dots,\,\,
       \partial_n:=
       \frac{\partial}{\partial x_n},
\]
that is, we set
\begin{align*}
 \hat {\vect{e}}_1&
 :=\partial_1,\quad \vect{e}_1
 :=\hat {\vect{e}}_1/|\hat {\vect{e}}_1|, \\
 \hat {\vect{e}}_j&:=\partial_j-\sum_{i=1}^{j-1} 
 \inner{\partial_j}{\vect{e}_i}\vect{e}_i,
 \quad
 \vect{e}_j:=\hat{\vect{e}}_j/|\hat{\vect{e}}_j| 
 \qquad (j=2,\dots,n-1).
\end{align*}
Then $\vect{e}_1,\dots,\vect{e}_{n-1}$ are smooth 
vector fields on $U$.
Finally, we set
\begin{equation}\label{eq:schmidt}
\hat {\vect{e}}_n:=\partial_n-\sum_{i=1}^{n} 
\inner{\partial_n}{\vect{e}_i}\vect{e}_i,
\quad
\vect{e}_n:=\frac{\hat{\vect{e}}_n}
{\lambda \prod_{j=1}^{n-1}|\hat {\vect{e}}_j|}, 
\end{equation}
where $\lambda$ is a $C^\infty$-function on $U$
satisfying \eqref{eq:lambda3}.
Then the resulting vector field $\vect{e}_n$ is defined only on
$U\setminus \Sigma^{n-1}$, and
$\vect{e}_1,\dots,\vect{e}_n$
consists of an orthonormal frame 
on $U\setminus \Sigma^{n-1}$, which is called
the {\it orthonormal frame field associated to\/} 
the adapted coordinate system $(x_1,\dots,x_n)$.

We now set
\begin{equation}\label{eq:omega_ij}
 \omega_{ij}:=\sum_{k=1}^n  
  \inner{e\nabla_{\partial_k}\vect{e}_j}{\vect{e}_i}\,
    dx_k
    =\sum_{k=1}^n \Gamma(\partial_k,\vect{e}_j,\vect{e}_i)\,dx_k
    \qquad (i,j=1,\dots,n),
\end{equation}
on $U\setminus \Sigma^{n-1}$,
where $\nabla$ is the Levi-Civita connection
of the metric $ds^2$ on $M^n\setminus \Sigma^{n-1}$
and $\Gamma$ is the Kossowski pseudo-connection.

\begin{lemma}
 \label{lem:kossowski}
 Each $\omega_{ij}$
 $(i,j=1,\dots,n)$ can be extended
 to a smooth $1$-form on $U$.
\end{lemma}

\begin{proof}
If $1\le i,j\le n-1$, then $\omega_{ij}$ is
trivially a smooth $1$-form on $U$. 
So we consider the case $i=n$.
By \eqref{eq:schmidt} and
\eqref{eq:omega_ij},
it holds on $M^{n}\setminus \Sigma^{n-1}$ that
\[
  \Gamma\left(
       \partial_k,\vect{e}_j,\hat{\vect{e}}_n
        \right)
            =\lambda \omega_{nj}(\partial_k)
       \prod_{l=1}^{n-1}|\hat {\vect{e}}_l|
     \qquad (k=1,\dots,n,~j=1,\dots,n-1).
\]
Since $ds^2$ is admissible, the left hand side vanishes
on $U\cap \Sigma^{n-1}$,
there exists a smooth function
$a_{kj}\in C^\infty(U)$ such that
\[
\Gamma\left(\partial_k,\vect{e}_j,\hat{\vect{e}}_n
\right)=
\lambda a_{kj}
\prod_{l=1}^{n-1}|\hat {\vect{e}}_l|
\qquad (k=1,\dots,n,~j=1,\dots,n-1).
\]
In particular, we have that
$\omega_{nj}(\partial_k)=a_{kj}$.
We next consider the case $j=n$.
Since
\[
  \Gamma(\partial_k,\hat{\vect{e}}_n,\vect{e}_i)
          =-\Gamma(\partial_k,\vect{e}_i,\hat{\vect{e}}_n)=0
\]
on $U\cap \Sigma^{n-1}$,
one can easily see that
$\omega_{in}(\partial_k)$ can also be
extended as a $C^\infty$-function on $U$.
Finally, $\omega_{nn}$ vanishes
on $U\setminus \Sigma^{n-1}$, and 
is trivially extended on $U$.
\end{proof}

\begin{proof}[Proof of Theorem \ref{thm:kossowski}]
 Let  $\{(U_a;x^a_1,\dots,x^a_n)
\}_{a\in \Lambda}$
be an atlas of $M^n$ consisting of
local adapted coordinate systems.
Since $ds^2$ is a Kossowski metric,
there exists a $C^\infty$-function 
$\lambda_a$ on $U_a$ ($a\in \Lambda$)
 such that
 \[
     \det(g^a_{ij})=(\lambda_a)^2,
 \]
 where $ds^2=\sum_{i,j=1}^n g^a_{ij}
dx_i^a dx_j^a$.

 We fix two indices $a,b\in \Lambda$
 such that $U_a\cap U_b\ne \emptyset$, and
 set
 \[
    (U;x_1,\dots,x_n):=(U_a;x^a_1,\dots,x^a_n)
 ,\qquad
    (V;y_1,\dots,y_n):=(U_b;x^b_1,\dots,x^b_n)
 \]
 for the sake of simplicity.
 We denote by
 $\vect{e}_1,\dots,\vect{e}_n$
 and
 $\tilde{\vect{e}}_1,\dots,\tilde {\vect{e}}_n$
 the orthonormal frame fields associated to 
 the adapted coordinate systems $(x_1,\dots,x_n)$
 and $(y_1,\dots,y_n)$, respectively.
 By the previous procedure of
 orthogonalization, there are 
 upper triangular
 matrices $\T$ an $\tilde {\T}$ 
 such that
 \begin{align}\label{eq:T10}
  (\vect{e}_1,\dots,\vect{e}_n)
    &=
     \left(
         \frac{\partial}{\partial x_1},
          \dots,  \frac{\partial}{\partial x_n}
         \right)
    \T, \\
    \label{eq:T20}
         (\tilde{\vect{e}}_1,\dots,\tilde{\vect{e}}_n)
      &=
   \left(
     \frac{\partial}{\partial y_1},
         \dots,  \frac{\partial}{\partial y_n}
         \right)
    \tilde{\T}.
 \end{align}
 These two matrices $\T$
 and $\tilde{\T}$ can be written as
 \begin{align}
  \label{eq:T1}
  \T &=\pmt{* & * \\
  \mb 0 & d  },\\
  \label{eq:T2}
  \tilde{\T}
  &=\pmt{* & * \\
  \mb 0 & \tilde d  },
 \end{align}
 where $*$ means 
 a real valued (or a matrix valued)
 function which is smooth along $\Sigma^{n-1}$
 and 
 $\mb 0$ is the row zero vector in $\R^{n-1}$.
 On the other hand, $d$ (resp.\ $\tilde d$) means a
 \lq divergent function\rq\ which is
 not smooth along 
 $U_a\cap \Sigma^{n-1}$ but $\lambda_a d$
 (resp.\ $\lambda_b \tilde d$)
 is a $C^\infty$-function on $U_a$
 (resp.\ $U_b$). 

 Since $\T$  and $\tilde {\T}$
 are upper triangular
 matrices, one can easily check that
 \begin{equation}\label{eq:ll}
  {\lambda_a}\det(\T)={*},
  \qquad
  {\lambda_b}\det(\tilde{\T})={*}.
 \end{equation}
 On the other hand, 
 there is a matrix valued
 function $\J$
 such that the equality
 \begin{equation}\label{eq:J0}
  \left(
    \frac{\partial}{\partial y_1},
    \dots,  \frac{\partial}{\partial y_n}
	\right)
  =
  \left(
   \frac{\partial}{\partial x_1},
   \dots,  \frac{\partial}{\partial x_n}
  \right)
 \J
 \end{equation}
 holds on $U_a\cap U_b$.
 Since $(x_1,\dots,x_n)$ and
 $(y_1,\dots,y_n)$ are adapted coordinate
 systems, we can write
 \begin{equation}\label{eq:J}
 \J =\pmt{* & \mb 0 \\
                  * & *  },
 \end{equation}
 where $\vect{0}$ is the column zero vector in $\R^{n-1}$.
 By \eqref{eq:T10},
 \eqref{eq:T20} and \eqref{eq:J0},
 we have that
 \[
 (\tilde{\vect e}_1,\dots,\tilde{\vect e}_n)
 =
 ({\vect e}_1,\dots,{\vect e}_n)
 {\T}^{-1}
 {\J}
 \tilde{\T}.
 \]

 We now compute ${\T}^{-1}{\J}\tilde{\T}$
 using the relations
 \[
 *+*=*,\qquad d + *= d,\qquad
 *\times *=*,\qquad d \times *= d
 \]
 on $U_a$ and
 \[
 *+*=*,\qquad \tilde d \times *= \tilde d,
 \qquad
 *\times *=*,\qquad \tilde d \times *= \tilde d
 \]
 on $U_b$, where $\times$ means
 the usual multiplications of
 scalars and matrices.
 These relations follow from the definitions
 of divergent terms $d$ and $\tilde d$. 
 Here, $d\times *$ might not be
 divergent in general. The above convention
 $d \times *= d$ means that
 $d \times *$ can be a divergent 
 term as a possibility.
 On the other hand, if the term $d\times d$ appears, 
 then it
 is more dangerous than the divergent terms,
 since $\lambda d\times d$ is still a divergent term.
 Fortunately, such a term never appears
 in the calculation of ${\T}^{-1}{\J}\tilde{\T}$ as follows:
 the equalities \eqref{eq:T1} 
 and \eqref{eq:ll} yield that
 \[
  \T^{-1} =\lambda_a
           \pmt{d & * \\
         \mb 0 & * }.
 \]
 We now set
 \[
 \tau_{ab}^{}:=
         {\T}^{-1}{\J}\tilde{\T},
 \]
 which gives a $C^\infty$-function 
 on $U_a\cap  V_b\setminus \Sigma^{n-1}$.
 Since
 \[
  {\T}^{-1}{\J}=
     \lambda_a
      \pmt{d & * \\
         \mb 0 & * }
       \pmt{* & \mb 0 \\
           * & * }
       =
     \lambda_a
        \pmt{d & *\\
             * & *},
 \]
 we have that
 \[
   \tau_{ab}=
          \lambda_a
           \pmt{d & *\\
                  * & *} 
           \pmt{* & *\\
              \mb 0 & \tilde d}
      =
         \lambda_a
     \pmt{d & d+\tilde d\\
     * & \tilde d}
      =
     \pmt{* & *\\
     * & *},
 \]
 since $\lambda_a=\lambda_b \times *$.
 So we can conclude that
 $\tau_{ab}$ can be smoothly
 extended on $U_a\cap U_b$.
 In particular, 
 the co-cycle condition
 \begin{equation}\label{eq:cocycle}
  \tau_{ab}\tau_{bc}\tau_{ca}=\id
 \end{equation}
 holds on $U_a\cap U_b\cap U_c$,
 where $\id$ is the identity matrix.
 Thus there exists a vector bundle $\E$ 
 with inner product $\inner{~}{~}$
 whose transition
 functions are $\{\tau_{ab}\}$.
 Let
 \[
   \Omega^a:=
       \pmt{\omega^a_{11}& \dots & \omega^a_{1n}\\
          \vdots & \ddots& \vdots \\
      \omega^a_{n1}& \dots & \omega^a_{nn}}
\]
 be a matrix valued $1$-from on $U_a$ according to
 Lemma \ref{lem:kossowski}, which gives a connection form
 of the Levi-Civita connection of $ds^2$ on 
 $U_a \setminus \Sigma^{n-1}$.
 In particular, $\Omega^a$ 
 takes value in the set of skew-symmetric matrices.
 The family of
 matrix valued 1-form 
 $\{\Omega^a\}_{a\in \Lambda}$
 satisfies the identity
 \begin{equation}\label{eq:conn}
  \Omega^b=
   \tau_{ab}^{-1}(d\tau_{ab}^{})
   +\tau_{ab}^{-1}
   \Omega^a \tau_{ab}^{}
 \end{equation}
 on $U_a\cap U_b\setminus \Sigma^{n-1}$.
 Then by continuity,
 \eqref{eq:conn}
 holds on $U_a\cap U_b$.
 Thus, it induces a metric connection $D$ on $\E$.
 By the definition of $\E$, the bundle homomorphism
 \[
  \phi:TM^{n}\longrightarrow \E
 \]
 is canonically induced so that
 $\phi(\vect{e}^a_1),\dots,\phi(\vect{e}^a_n)$
 consists of an orthonormal frame of $\E$
 on $U_a \setminus \Sigma^{n-1}$,
 where
 $\vect{e}^a_1$, \dots, $\vect{e}^a_n$
 is the orthonormal frame field associated
 to $(x_1^a,\dots,x_n^a)$.
 Then the restriction of the map $\phi$
 into $M^n\setminus \Sigma^{n-1}$ gives
 a vector bundle isomorphism between
 the tangent bundle of $M^n\setminus \Sigma^{n-1}$
 and $\E|_{M^n\setminus \Sigma^{n-1}}$, 
 and  $\phi^*\inner{~}{~}=ds^2$ 
 holds on $M^n\setminus \Sigma^{n-1}$.
 Then by continuity,
 $\phi^*\inner{~}{~}=ds^2$ 
 holds on all $M^n$.
 On the other hand,
 the pull-back connection of $D$
 coincides with the Levi-Civita connection of
 $ds^2$ on $M^n\setminus \Sigma^{n-1}$.
 In particular,
 \eqref{eq:torsion} holds on $M^n\setminus \Sigma^{n-1}$.
 Then, by continuity,
 \eqref{eq:torsion} also holds on all of $M^n$.
 Thus we get a coherent tangent bundle associated to
 the Kossowski metric $ds^2$.
\end{proof}

A Kossowski metric is said to be
{\it co-orientable\/} if one can choose the chart
\[
   \{(U_a;x^a_1,\dots,x^a_n)
       \}_{a\in \Lambda} 
\]
of $M^n$
such that
\[
   \mu:=\lambda_a dx^a_1\wedge
               \dots\wedge dx^a_n
\]
gives a globally defined smooth $n$-form on $M^n$.
It can be easily checked that the co-orientability of $ds^2$
corresponds to the fact that the induced bundle $\E$
is orientable (cf.\ \cite[Prop.2.11]{HHNSUY}).
We remark that each $\lambda_a$
($a\in \Lambda$) is a $\phi$-function
of the induced coherent tangent bundle.

\begin{definition}
 A Kossowski metric $ds^2$ on $M^n$
 is called  
 a {\it Morin metric\/} if its
 induced coherent tangent bundle
 admits only $A_{k+1}$-points ($k=1,\dots,n$).
\end{definition}

Then as an application of the formula \eqref{eq:main},
we get the following assertion.

\begin{corollary}\label{cor:morin-metric}
 Let $ds^2$ be a co-orientable Morin metric defined 
 on an oriented compact manifold $M^{2m}$.
 Then the identity \eqref{eq:main} holds,
 where $\chi_\E$ is the Euler characteristic of the
 coherent tangent bundle
 $\E$ associated to $ds^2$.
\end{corollary}

This corollary is a generalization 
of \cite[Prop. 3.3]{HHNSUY}.
The following assertion is 
the spacial case of this corollary
if we set $ds^2$ to be
the dual metric of the conformally
flat metric as in Example \ref{ex:conf}.

\begin{theorem}

\label{thm:C}
 Let $(M^{2m},g)$ be a compact orientable
 conformally flat manifold whose dual
 conformally flat metric $\check g$
 admits only $\A_k$-singularities
 for $2\leq k\leq 2m+1$.
 Then the singular set of the dual conformally 
 flat metric $\check g$  satisfies \eqref{eq:BW},
 where $M^{2m}_+$ {\rm(}resp.\ $M^{2m}_-${\rm)} is
 the 
 subset of $M^{2m}$ at which the determinant
 of the Schouten tensor
 is positive {\rm(}resp.\ negative{\rm)},  and $\AA^+_{2j+1}$
 {\rm(}resp.\ $\AA^-_{2j+1}${\rm)}
 is the set of positive {\rm(}resp.\ negative{\rm)}
 $\A_{2j+1}$-points $(j=1,\dots,m)$
 of the bundle homomorphism $\check B$.
\end{theorem}

\appendix
\section{Extension of generic vector fields}

We prove the following assertion, which is needed to
prove the existence of a characteristic vector field
associated to a given Morin homomorphism:

\begin{lemma}\label{lem:app}
 Let $M^n$ be a compact manifold and
 $X$ a $C^\infty$-vector field defined on
 an open subset of $M^n$
 containing a compact subset $K$
 such that $X$ has no zeros on the boundary $\partial K$
 of $K$.
 Then there exists a $C^\infty$-vector field
 $\tilde X$ defined on $M^n$ such that
 $\tilde X$ coincides with $X$ on $K$
 and has only generic zeros on $M^n\setminus K$.
\end{lemma}

\begin{proof}
 We may assume that $X$ is defined on a 
 neighborhood $U$ of $K$.
 Take an open subset $V$ such that
 \[
    K\subset V\subset \overline V\subset U,
 \]
 where $\overline V$ is the closure of $V$.
 Taking $U$ sufficiently close to $K$, we may assume that
 $X$ has no zeros on $U\setminus K^\circ$, where
 $K^\circ$ (possibly empty) 
 is the set of the interior points of $K$.
 Then we can take $C^\infty$-functions
 $\rho_j:M^n\to [0,1]$ $(j=1,2)$
 such that $\rho_1=1$ on $K$
 (resp.\ $\rho_2=1$ on $\overline V$)
 and $\rho_1=0$  on $M^n\setminus V$
 (resp.\ $\rho_2=0$ on $M\setminus U$).
 We set
 $\hat X:=\rho_2 X$,
 which  is a vector field on $M^n$.
 It is well-known that there exists a
 sequence of generic 
 vector fields $\{Z_j\}_{j=1,2,3,\dots}$
 on $M^n$
 converging to $\hat X$ with respect to
 the Whitney $C^\infty$-topology.
 We set
 \[
    \tilde X_j:=\rho_1 \hat X+(1-\rho_1)Z_j.
 \]
 Then
 $\tilde X_j$ coincides with $X$ on $K$,
 because $\rho_1=\rho_2=1$ on $K$.
 Since 
 $\hat X$ has no zeros on 
 the compact set $\overline V\setminus K^\circ$,
 $\tilde X_j$ has a zero at $p\in \overline V\setminus K^{\circ}$
 if $\hat X=-\frac{1-\rho_1}{\rho_1}Z_j$ holds at $p$.
 This is impossible for sufficient large $j$,
 since $Z_j\to \hat X$ as $j\to \infty$
 and $\rho_1\in [0,1]$.
 Moreover,
 $\tilde X_j$ coincides with $Z_j$
 on $M^n\setminus V$, since $\rho_1=0$ on
 the complement of $V$.
 Thus it has only generic zeros on
 $M^n\setminus V$.
 In particular, $\tilde X_j$ has 
 the desired property for sufficiently large $j$.
\end{proof}

\begin{acknowledgements}
 The authors thank Mitsutaka Murayama, Toshizumi Fukui
 and Kazuhiro Sakuma for valuable comments.
 They also thank the referee for
 careful reading and valuable comments.
\end{acknowledgements}

\end{document}